\documentclass[final,onefignum,onetabnum]{siamart190516}

\usepackage{lipsum}
\usepackage{amsfonts}
\usepackage{graphicx}
\usepackage{epstopdf}
\usepackage{algorithmic}
\usepackage{amssymb}
\usepackage{mathrsfs}  

\ifpdf
  \DeclareGraphicsExtensions{.eps,.pdf,.png,.jpg}
\else
  \DeclareGraphicsExtensions{.eps}
\fi

\newsiamremark{remark}{Remark}
\newsiamremark{example}{Example}
\newsiamremark{hypothesis}{Hypothesis}
\crefname{hypothesis}{Hypothesis}{Hypotheses}
\newsiamthm{claim}{Claim}

\headers{SAA Feasibility via VC Dimension}{H. Lam and F. Li}

\title{General Feasibility Bounds for Sample Average Approximation via Vapnik-Chervonenkis Dimension\thanks{Submitted.
\funding{This work was partially supported by the National Science Foundation under grants CAREER CMMI-1834710 and IIS-1849280.}}}

\author{ Henry Lam \thanks{Department of Industrial Engineering and Operations Research, Columbia University, New York, NY 10027 
  (\email{khl2114@columbia.edu}).} 
  \and Fengpei Li \thanks{Corresponding author. Department of Industrial Engineering and Operations Research, Columbia University, New York, NY 10027. Current address: Morgan Stanley Machine Learning Research, New York, NY 10019 (\email{fl2412@columbia.edu}).} }

\usepackage{amsopn}

\ifpdf
\hypersetup{
  pdftitle={An Example Article},
  pdfauthor={D. Doe, P. T. Frank, and J. E. Smith}
}
\fi

\begin{document}

\maketitle

\begin{abstract}
  We investigate the feasibility of sample average approximation (SAA) for general stochastic optimization problems, including two-stage stochastic programs without relatively complete recourse. We utilize results from the \textit{Vapnik-Chervonenkis} (VC) dimension and \textit{Probably Approximately Correct} (PAC) learning to provide a general framework to construct feasibility bounds for SAA solutions under minimal structural or distributional assumption. We show that, as long as the hypothesis class formed by the feasible region has a finite VC dimension, the infeasibility of SAA solutions decreases exponentially with computable rates and explicitly identifiable accompanying constants. We demonstrate how our bounds apply more generally and competitively compared to existing results. 
  
\end{abstract}

\begin{keywords}
  sample average approximation, feasibility, sample complexity, two-stage stochastic programming, Vapnik-Chervonenkis dimension
\end{keywords}

\begin{AMS}
  90C15, 03E75
\end{AMS}

\section{Introduction}
Consider the stochastic optimization problem
\begin{equation}\label{sp}
\inf_{x\in\mathcal X} F(x) \triangleq \mathbb E [f(\xi,x)],
\end{equation}
where $\mathcal X $, the space for decision variables, is a nonempty closed subset of a Polish space, i.e., a complete, separable, metric space (e.g., $\mathcal X \subseteq \mathbb R^n$ or $\mathbb R^{n-p}\times \mathbb Z^{p}$ for mixed-integer decision sets, with Euclidean metric) equipped with the Borel $\sigma$-algebra $\mathscr B$ and $\xi:\Omega \rightarrow \Xi\subseteq \mathbb R^{r}$ is some random vector on a complete probability space $(\Omega,\mathcal F,\mathbb P)$. For each realization of $\xi(\omega)\in\Xi$ (henceforth we suppress the dependence of $\xi$ on $\omega\in\Omega$), we assume $f(\xi,\cdot) :\mathcal X\rightarrow \mathbb R \cup \{+\infty\}$ is a lower semicontinuous function mapping to the extended real line and the function $f(\xi,x)$ is measurable with resepct to the completion of product $\sigma$-algebra $\mathcal F \otimes \mathscr B$. We also assume the set $\{x: x\in\mathcal X \text{ and } F(x)<+\infty\}$ is non-empty. Moreover, we assume all quantities of interest are measurable and defer the technical arguments on measurability to the Appendix.

The class of problems under \eqref{sp} is difficult to evaluate in general, especially for high-dimensional $\xi$.  As a popular tractable approximation, the sample average approximation (SAA) method \cite{shapiro2014lectures} solves the sampling-based counterpart of \eqref{sp}:
\begin{equation}\label{saa}
\inf_{x\in\mathcal X} \hat F_N(x) \triangleq \frac{1}{N}\sum_{i=1}^N f(\xi_i, x),
\end{equation}
where $\xi^{[N]}\triangleq (\xi_1,\xi_2,...,\xi_N) $ are IID samples drawn from $\mathbb P$. An optimal solution of the SAA depends on the realization of $\xi^{[N]}$ and shall be denoted $x^\star(\xi^{[N]})$.  Theoretical properties and numerical performances of SAA have been extensively studied in, e.g., \cite{linderoth2006empirical,shapiro2014lectures,higle2013stochastic}, and its applications in stochastic optimization and chance-constrained programming can be found in, e.g., \cite{jirutitijaroen2008reliability,pagnoncelli2009sample, wang2008sample}. As an important class of \eqref{sp}, two-stage stochastic programming has applications in transportation planning \cite{barbarosoglu2004two,liu2009two}, disaster management \cite{noyan2012risk}, water recourse management \cite{huang2000inexact} and inventory management \cite{dillon2017two}. Most of these studies rely on the assumption that $f(\xi, x)<+\infty$ with probability 1 for all $x\in\mathcal X$, which is called the relatively complete recourse condition. However, in many real-world applications, this condition is restrictive and there has been a growing literature studying two-stage stochastic programming without this assumption, i.e., $\mathbb P(\{\xi\in \Xi:f(\xi, x)=+\infty\})>0$ for some $x\in\mathcal X$ (see \cite{chen2019sample, chen2019convergence, liu2019feasibility}). In such setting, the SAA solution $x^\star(\xi^{[N]})$ is not necessarily feasible for the original problem \eqref{sp} and it would be desirable to quantify the level of infeasibility for the SAA solution.

As nicely discussed in the recent works of \cite{chen2019sample, liu2019feasibility}, the feasibility issue of SAA arises when $f(\xi,\cdot)$ maps to the extended real line. Following the notations in \cite{liu2019feasibility}, we define $\text{dom } f_\xi\triangleq\{x\in\mathcal X: f(\xi, x) < +\infty \}$. Then by solving \eqref{saa} we would obtain an optimal solution $x^\star(\xi^{[N]}) \in \text{ dom } \hat {F}_N \triangleq \bigcap_{i=1}^N \text{ dom }f_{\xi_i}$ where $\text{dom }\hat {F}_N $ is the feasible region for the SAA. As we have assumed the set $\{x:x\in\mathcal X \text{ and } F(x)<+\infty\}$ is non-empty, the SAA feasible region is non-empty with probability 1 (since $F(x)<+\infty$ implies $\mathbb P(f(x,\xi)=+\infty)=0$ and in turn $x\in\text{dom } \hat F_N \text{ with probability 1}$). However, $x^\star (\xi^{[N]})$ may not be feasible for the original problem \eqref{sp}, i.e., $x^\star (\xi^{[N]}) \notin \text{dom } F \triangleq \{x\in\mathcal X: F(x) < +\infty\}$. In other words, defining the violation probability $V(x)$ for $x\in\mathcal X$ as
\begin{equation}
V(x)\triangleq \mathbb P(\xi\in\Xi: x\notin \text{dom }f_\xi),
\end{equation}
we could have $V(x^\star (\xi^{[N]}))>0$ with positive probability.

In this paper, we utilize a framework based on the Vapnik-Chervonenkis (VC) dimension to analyze the feasibility of SAA solutions, including two-stage stochastic programming. Following  \cite{chen2019sample,liu2019feasibility}, we focus on showing the exponential decrease of $V(x^\star (\xi^{[N]}))$ as $N$ grows. Specifically, letting $\mathbb P^N$ denote the sampling measure governing the generation of IID samples $\xi^{[N]}$ (notice the feasibility of $x^\star (\xi^{[N]})$ is random depending on $\xi^{[N]}$), we derive exponential bounds for $V(x^\star(\xi^{[N]}))$ under $\mathbb P^N$. As a key contribution, we show that, when the VC dimension of the feasible domain is finite, our framework produces feasibility bounds that are both general and explicit. In particular, the constants in the bounds are computable with respect to problem parameters. Moreover, aside from finite VC dimension of the feasible domain, we impose no additional requirement on the geometric or distributional properties of \eqref{sp} (i.e., whether the problem is convex or linear, whether the optimal solution is in the interior or on the boundary of $\text{dom } F$, whether $\mathcal X$ is finite or functions $\{f(\xi,\cdot)\}_{\xi\in\Xi}$ have a chain-constrained domain, as in \cite{liu2019feasibility}), or specific regularity conditions on $f(\xi, x)$ (i.e., Lipschitz continuity or the existence of certain moment generating function, as in \cite{chen2019sample, liu2019feasibility}). Consequently, the analysis is widely applicable in both scenarios where some of the best-known results on SAA feasibility have been presented, and other scenarios where no related results have been established. Furthermore, the feasibility result under this framework is not restricted to the optimal solution of SAA, but any generic point within the SAA feasible region with probability 1. As a result, when the SAA problem is non-convex and solvable only up to local optimum, or when approximate algorithms are required, our feasibility guarantee would still hold. Finally, we show that the generality of this framework does not come at a cost of worse sample complexity as our bounds are comparable to, if not better than existing ones. 

There exists a large body of literature which leverages techniques from statistical/computational learning theory and empirical process theory  to analyze the performance of SAA, including (asymptotic) convergence rate of optimal solution and (non-asymptotic)  concentration inequalities. The seminal work of \cite{shapiro2005complexity} surveys the sample complexity of two-stage stochastic (linear) programming in obtaining so-called $\epsilon$-optimal solutions using large deviations theory or uniform law of large numbers. In \cite{oliveira2017sample1, oliveira2017sample2}, Talagrand’s ``generic chaining" tail bounds and ``localization” techniques for sub-Gaussian processes (see also \cite{talagrand2014upper}) are used to derive non-asymptotic risk bounds for the SAA optimal solution. Moreover, in \cite{ermoliev2013sample}, the Rademacher averages of functional class, a concept closely related to the VC dimension, is used to established convergence results for SAA in general compound stochastic optimization
problems. As we can see, the concept of using statistical learning theory/complexity to analyze SAA is not entirely new. However, these results largely focus on the optimization accuracy of SAA as in \cite{shapiro2005complexity}, or the convergence rate of optimal SAA solution. In contrast, our paper provides a first application of the VC dimension and PAC-learning in analyzing the feasibility of SAA solutions (optimal or otherwise). To this end, a related application is studied in \cite{de2004constraint,erdougan2006ambiguous} under the context of chance-constrained programming, where the feasibility of constraint sampling is analyzed. It is worth noting that, the Rademacher complexity theory, as an alternative to the VC theory, typically obtain tighter (usually by a logarithmic factor) generalization error bounds which in fact subsume bounds provided by the VC theory \cite{kaariainen2004relating}. However, the standard
method for evaluating Rademacher complexity relies on certain empirical risk minimization (ERM) algorithm \cite{kaariainen2004relating} which may not be tractable, and one common way to bound the data-dependent Rademacher complexity is directly by the VC dimension. On the other hand, for the purpose of analyzing SAA feasibility, VC bounds are sufficient to establish results that are either new or comparable/superior to known ones. More importantly, the bounds we achieve do not depend on any data distribution, and carry explicit constants rather than a standalone big-$O$ notation as in typical prior results on SAA feasibility or learning theory analysis. Thus, in this paper we focus on VC bounds which, although not the tightest possible (we do remark on where the bounds can be improved when appropriate), offer a great generality and a wide applicability, while maintaining a relatively simple, concise appeal.
 
 Finally, we mention a growing body of literature exploiting low-dimensional structures (i.e., sparsity or low rankness) in high-dimensional problems to achieve a sample complexity that is logarithmic-in-dimension, instead of the usual linear-in-dimension sample complexity for the optimality of SAA solutions \cite{bugg2021logarithmic, liu2019regularized, liu2019sample}. As we shall see, these improvements correspond to reduced VC dimensions for low-dimensional structures and also apply to our analysis of SAA feasibility. 

The rest of the paper is organized as follows. In Section 2, we review the recent papers with closely related results, especially \cite{chen2019sample,liu2019feasibility}. In Section 3, we present the concepts of our framework and main result. Section 4 focuses on the applications of our main result. In the first part of this section, we present three novel results on the feasibility of two-stage stochastic programs, first a general feasibility result (not limited to linear ones), second a feasibility bound for mixed-integer recourse, and third an exploitation of low-dimensional structures to obtain a logarithmic-in-dimension sample complexity in high-dimensional two-stage stochastic programs. In the second part of Section 4, we consider two special structures used to provide feasibility results in the existing literature with whom we compare our results. The first is the chain-constrained domain considered in \cite{liu2019feasibility}. The second is a finite feasible domain considered in \cite{chen2019sample}. Finally, in the Appendix, we provide a specially constructed example that allows a direct verification of our feasibility results.


\section{Review of Related Results}
We discuss the existing results on SAA feasibility in \cite{chen2019sample, liu2019feasibility}. A considerable part of $\cite{chen2019sample}$ discusses how to solve a so-called ``padded", modified version of SAA to obtain a \textit{complete feasible} solution (i.e. $V(x)=0$) with high confidence, which is somewhat different from the perspectives of our paper and \cite{liu2019feasibility}. In particular, we consider the feasibility for SAA in its original form and do not restrict our attention to complete feasible solutions. Both \cite{chen2019sample} and \cite{liu2019feasibility} discuss results of the form
\begin{equation}
	\mathbb P^N(V(x^\star(\xi^{[N]}))>\epsilon) \leq \delta,
\end{equation}
referred to as \textit{high recourse likeihood solution} in \cite{chen2019sample}. In particular, \cite{chen2019sample} presents these bounds in two cases, one for finite $\mathcal X$ and another for two-stage stochastic linear programs. We shall discuss them in detail in Section 4 when we compare with our results. On the other hand, the feasibility results in \cite{liu2019feasibility} are more general and can be summarized into three different scenarios.
\begin{itemize}
	\item Scenario 1: In the presence of the so-called chain-constrained domain of order $m$ (to be explained later) on $\text{dom } f_\xi$, \cite{liu2019feasibility} shows 
\begin{align*}
	\mathbb P^N(V(x^\star(\xi^{[N]}))>\epsilon) &\leq \sum_{k=0}^{m-1}  {N \choose k} \epsilon^k(1-\epsilon)^{N-k} \nonumber\\
	&\leq \exp\bigg\{-\frac{(N\epsilon-m+1)^2}{2N\epsilon}\bigg\}.
\end{align*}
The second inequality is given by the Chernoff bound, also shown in \cite{campi2011sampling} and \cite{liu2019feasibility}.

\item Scenario 2: In the context of  convexity, meaning $\mathcal X$ is closed and convex and the set of optimal solutions $\mathcal X^\star$ is non-empty, and $f(\xi, \cdot)$ is convex for all $\xi\in\Xi$, along with additional regularity conditions on $f(\xi, \cdot)$ and $\mathcal X$ (e.g., finite moment generating function for $f(\xi,x)$), \cite{liu2019feasibility} shows that for $\mathcal X^\star$ in the interior of $\text{dom } F$, 
\begin{equation*}
	\mathbb P^N(V(x^\star(\xi^{[N]}))>0) \leq Ce^{-N\beta},
\end{equation*}
where $C$ and $\beta$ are unknown constants.

\item Scenario 3:
In the context of  convexity, if $\text{dom } f_\xi$ is a chain-constrained domain as in Scenario 1, along with the additional regularity conditions, \cite{liu2019feasibility} shows that for $\mathcal X^\star$ which may have non-empty intersection with the boundary of $\text{dom } F$, one can still guarantee
\begin{align*}
	\mathbb P^N(V(x^\star(\xi^{[N]}))>\epsilon) &\leq  Ce^{-N\beta}+\sum_{k=0}^{|\mathcal J|-1}  {N \choose k} \epsilon^k(1-\epsilon)^{N-k} \nonumber\\
	&\leq Ce^{-N\beta}+\text{exp} \bigg\{-\frac{(N\epsilon-|\mathcal J|+1)^2}{2N\epsilon}\bigg\}
\end{align*}
where $C$ and $\beta$ are again unknown constants as in Scenario 2 and $\mathcal J$ is the index set of active constraints at $\mathcal X^\star$ with the boundary of $\text{dom } F$. Notice it is shown in \cite{campi2008exact} that $|\mathcal J|$ is bounded by $n$, the dimension of the decision variable, which yields a useful upper bound regardless of the behavior of $\mathcal J$ (Also note that in this case the order of the chain-constrained domain does not show up explicitly in the bound).
\end{itemize}

In all scenarios, a desirable exponential decrease of $V(\cdot)$ as $N$ grows can be shown. However, there are several potential limitations. First, there exist hidden constants in the feasibility bound: In Scenarios 2 and 3, which are of great importance in stochastic convex programming, the rate of exponential decrease is established but governed by unknown constants $C$ and, more importantly $\beta$ which directly dictates the rate of exponential decrease. Second, the dependence of the bound on $m$, the order of the chain-constrained structure, can become potentially restrictive as $m$ gets large (or even infinite) in many cases, a challenge that is also mentioned in \cite{chen2019sample}.  Furthermore, even though it is motivated from practical examples in \cite{liu2019feasibility}, the chain-constrained structure can be difficult to verify in general. Notably, the feasibility bound derived in Scenario 3 (the optimal solution of \eqref{sp} intersects the boundary of $\text{dom } F$) is not explicitly dependent on the chain order $m$, but the chain-constrained structure is still required for the analysis in \cite{liu2019feasibility}. Finally, note that while an explicit bound (all constants known or verifiable) is presented in Scenario 1, it is a feasibility bound on the entire $\text{dom } \hat F_N$ instead of just $x^\star(\xi^{[N]})$, and is still under the chain-constrained domain assumption. Due to all these limitations, it is desirable to generalize the feasibility bound beyond the chain-constrained domain setting and with explicit constants.

\section{Framework and Main Results}
In this section we review the VC dimension framework and introduce our main results. In particular, we are interested in the VC dimension of a collection of subsets on $\Xi$. This approach gives bounds for any generic point in $\text{dom } \hat F_N$, the feasible region of the SAA, which in particular implies bounds for $x^\star(\xi^{[N]})$. Moreover, 
instead of looking at $\text{dom } f_\xi=\{x\in\mathcal X: f(\xi, x) < +\infty \}$, we investigate
\begin{equation}\label{domm}
 H_x\triangleq\{\xi\in\Xi: f(\xi,x)<+\infty \}  \text{\ \ for\ } x\in\mathcal X
\end{equation}
and define the class of subsets
\begin{equation*}\label{dommm}
	\mathcal H\triangleq \{H_x\}_{x\in\mathcal X} \cup \{\Xi\}.
\end{equation*}

Consider the VC dimension of $\mathcal H$. The VC dimension is commonly used to describe the complexity of a collection of sets or functions \cite{anthony1997computational,kearns1994introduction,vapnik2013nature}, which is also known as the ``hypothesis space" in computational learning theory. The concept applies to a class of subsets $\mathcal H$ (see \cite{van2009note}), and can be naturally generalized to binary functions. To define the VC dimension of a class of subsets $\mathcal H$, first we say a set of points $\{x_1,...,x_d\}$ is ``shattered" by $\mathcal H$ if any subset of $\{x_1,...,x_d\}$ can be ``picked out" by some element $C\in \mathcal H$, meaning that for any subset $D \subseteq \{x_1,...,x_d\}$, there is a set $C \in\mathcal  H$ such that $D \subseteq C$ and $(\{x_1,...,x_d\}\setminus D) \cap C =\emptyset$. The VC dimension of $\mathcal H$ is defined to be the maximal cardinality of the sets it can shatter, denoted by $d_{VC}(\mathcal H)$. For example, some classical results on the VC dimensions are

\begin{itemize}

	\item Positive intervals: If $\mathcal H=\big\{\big\{x\in\mathbb R: x\in [a,b]\big\}|0\leq a\leq b\big\}$, we have $d_{VC}(\mathcal H)=2$.

		\item Affine hyperplanes (Perceptrons): If $\mathcal H=\big\{\big\{ x\in\mathbb R^d : a^T x+b \geq 0 \big\}| a\in\mathbb R^d, b\in \mathbb R\big\}$, we have $d_{VC}(\mathcal H)=d+1$. If $b$ is fixed to be 0, then $d_{VC}(\mathcal H)=d$.
	\item Convex sets: If $\mathcal H=\big\{C: C \subseteq \mathbb R^d \text{ and } C \text{ is convex}\big\}$ and $d\geq 2$, then $d_{VC}(\mathcal H)=+\infty$.

\end{itemize}
An important concept in computational learning theory tightly related to the VC dimension is \textit{Probably Approximately Correct} (PAC) learning (see, e.g. \cite{kearns1994introduction}). In this context, the VC dimension of $\mathcal H$ can be used to derive bounds on the sample complexity needed to achieve a desired level of accuracy between ``in-sample-error" and ``generalization error" within class $\mathcal H$ (see, e.g. \cite{anthony1997computational,blumer1989learnability,kearns1994introduction}). As it turns out, this type of results can transfer towards sample complexity results on the feasibility of SAA solutions.

 Moreover, we note that the $\Xi$ in \eqref{domm} can be reparametrized and does not have to be viewed in $\mathbb R^r$ for fixed $r$ defined in \eqref{sp}. For illustration, consider the following example.

 \begin{example}\label{example}
  Suppose $x\in\mathcal X \subseteq \mathbb R$ and $\xi$ is a random vector defined on $\mathbb R^r$ for some $r>0$. Let $g(\cdot): \mathbb R^r \rightarrow \mathbb R$ be a given function. Then, suppose $f(\xi,x)<+\infty$ if and only if $g(\xi)\cdot x\geq 1$. Then, $H_x$ in \eqref{domm} becomes
  \begin{equation*}
      H_x=\{\xi: g(\xi)\cdot x \geq 1\}\subseteq \mathbb R^r, \forall x \in \mathcal X.
  \end{equation*}
 On the other hand, if we define random variable $\xi'=g(\xi)$ on $\mathbb R$, then we can alternatively define
 \begin{equation*}
      H'_x=\{\xi': \xi'\cdot x \geq 1\}\subseteq \mathbb R, \forall x \in \mathcal X.
  \end{equation*}
 Consequently, instead of fixing a canonical representation of $\xi$ in \eqref{sp}, we sometimes utilize this flexibility to change representations for convenience.
 \end{example}
 
 Finally, we clarify the use of notation $[\cdot]$. For a positive integer $q$, $[q]$ denotes the set $\{1,..., q\}$. For a vector $ \boldsymbol v\in\mathbb R^q$, $[\boldsymbol v]_j$ denotes the $j$-th component of $\boldsymbol v$, for $j\in[q]$, and $\|\boldsymbol v\|_0$ denotes the number of non-zero entries of $\boldsymbol v$.

\subsection{Main Result}
We now present our main theorem on SAA feasibility and its proof.

\begin{theorem}\label{main}
	Let $\mathcal H\triangleq \{H_x\}_{x\in\mathcal X}\cup\{\Xi\}$ be the class of subsets defined in \eqref{domm} and $\xi^{[N]}=\{\xi_1,...,\xi_N\}$ be IID samples from $\mathbb P$ (consequently $\xi^{[N]}\sim \mathbb P^N$). Assume $\mathcal H$ has finite VC dimension $d$ (i.e., $d_{VC}(\mathcal H)=d<+\infty$). Moreover, assume the feasible region of SAA (i.e., $\text{dom } \hat F_N=\cap_{i=1}^N \{x\in\mathcal X: f(\xi_i,x)<+\infty\}$) is non-empty with probability 1 and $x^\star (\xi^{[N]})\in\mathcal X$ is the output of any algorithm that is guaranteed to be within the feasible region of SAA. Then, if  
	\begin{equation}\label{sc}
		N \geq \frac{4}{\epsilon}\bigg(d\log\big(\frac{12}{\epsilon}\big)+\log\big(\frac{2}{\delta}\big)\bigg),
	\end{equation}
	we have, for any $0<\delta,\epsilon <1$,
		\begin{equation}\label{meaissue}
\mathbb P^N \bigg(\sup_{x\in\text{dom } \hat F_N} V(x) > \epsilon\bigg) \leq  \delta,
	\end{equation}
and \footnote{The measurability of \eqref{meaissue} and \eqref{mea2} is discussed in the Appendix.} consequently 
	\begin{equation}\label{mea2}
\mathbb P^N\bigg(V(x^\star (\xi^{[N]}))>\epsilon\bigg) \leq  \delta.
	\end{equation}
\end{theorem}

\begin{proof}
	We first adopt the following notations from \cite{anthony1997computational}. For $x\in\mathcal X$, define $ {h}_x:\Xi\rightarrow \{0,1\}$ as
\begin{equation*}
   h_x(\xi)\triangleq\mathbf{1}_{H_x}(\xi) = \begin{cases}
1 \text {\ \ when } f(\xi,x)<+\infty\\
0 \text {\ \ when } f(\xi,x)=+\infty
\end{cases}
\end{equation*}
and denote 
\begin{equation*}
   \text{H}\triangleq\{h_x\}_{x\in\mathcal X}\cup \{\mathbf{1}_{\Xi}(\cdot)\},
\end{equation*}
to be a hypothesis space as in the PAC learning context. Now, given a ``target concept" $t\in\text{H}$, the error of a hypothesis $h\in\text{H}$ is defined in \cite{anthony1997computational} as
\begin{equation}\label{a1}
    \text{er}_{\mathbb P}(h)\triangleq \mathbb P(\{\xi\in\Xi: h(\xi)\neq t(\xi)\}).
\end{equation}
Building on \eqref{a1},  \cite{anthony1997computational} further defines
\begin{align}\label{a2}
    \text{H}[\xi^{[N]}]=&\{h\in\text{H} \text{ } |\text{ } h(\xi_i)=t(\xi_i), \text{ } \forall i \in [N]\} \nonumber\\
    B_\epsilon=& \{h\in\text{H} \text{ }| \text{ } \text{er}_\mathbb P (h) \geq \epsilon\}
\end{align}
It is straightforward to see that $\text{H}$ has the same VC dimension as $\mathcal{H}$ ($d_{VC}(\mathcal H)=d$). Thus, given $\text{H}$ with VC dimension $d<+\infty$ and let the target function be
\begin{equation}\label{target}
    t(\cdot)=\mathbf{1}_{\Xi}(\cdot)\in\text{H},
\end{equation}
Theorem 8.4.1 in \cite{anthony1997computational} states that if
\begin{equation*}
    N \geq \frac{4}{\epsilon}\bigg(d\log\big(\frac{12}{\epsilon}\big)+\log\big(\frac{2}{\delta}\big)\bigg),
\end{equation*}
we have
\begin{equation}\label{a3}
    \mathbb P^{N}(\{\xi^{[N]}\text{ }| \text{ } \text{H}[\xi^{[N]}] \cap B_\epsilon \neq \emptyset\})<\delta.
\end{equation}
Given samples $\xi^{[N]}$, we know from \eqref{a2} and \eqref{target} that the feasible region of SAA is 
\begin{align}\label{t1}
\text{dom } \hat F_N=&\cap_{i=1}^N \{x\in\mathcal X \text{ }| \text{ } f(\xi_i,x)<+\infty\}\nonumber\\
=&\{x\in\mathcal X \text{ }| \text{ } h_x(\xi_i)=1,\text{ } \forall i \in [N] \}\nonumber\\
=&\{x\in\mathcal X \text{ }| \text{ } h_x \in \text{H}[\xi^{[N]}] \}
\end{align}
and we also know from \eqref{a1} and \eqref{a2} that,
\begin{align}\label{t2}
V(x)\geq \epsilon \iff \text{er}_{\mathbb P} (h_x)\geq \epsilon \iff h_x \in B_\epsilon
\end{align}
for any $x\in\mathcal X$. It then follows from \eqref{t1} and \eqref{t2} that, given samples $\xi^{[N]}$,
\begin{equation}\label{t3}
    \{ x\in\text{dom } \hat F_N \text{ }| \text{ } V(x) \geq \epsilon\} \neq \emptyset \implies \text{H}[\xi^{[N]}] \cap B_\epsilon \neq \emptyset .
\end{equation}
Thus, if  $x^\star(\xi^{[N]})\in\text{dom } \hat F_N$ with probability 1, it follows from \eqref{a3} and \eqref{t3} that
\begin{align*}
  \mathbb P^{N}\bigg(\text{H}[\xi^{[N]}] \cap B_\epsilon = \emptyset\bigg)> 1-\delta \implies& \mathbb P^{N} \bigg(\{ x\in\text{dom } \hat F_N \text{ }| \text{ } V(x) \geq \epsilon\} = \emptyset \bigg)  > 1-\delta\nonumber\\
   \implies & \mathbb P^N \bigg(\sup_{x\in\text{dom } \hat F_N} V(x) \leq \epsilon\bigg) > 1-\delta,
   \nonumber\\
   \implies& \mathbb P^N\bigg(V(x^\star(\xi^{[N]}))\leq\epsilon\bigg)>1-\delta
\end{align*}
which concludes the proof.
\end{proof}
\begin{remark}\label{t1r}
First, in the proof of \cref{main}, the sample complexity in \eqref{sc} comes from Theorem 8.4.1 of \cite{anthony1997computational} and provides a $O(\frac{d}{\epsilon}\log(\frac{1}{\epsilon})+\frac{1}{\epsilon}\log(\frac{1}{\delta}))$ bound. It is worth noting that a better sample complexity of $O(\frac{d}{\epsilon}+\frac{1}{\epsilon}\log(\frac{1}{\delta}))$ can be achieved by recent breakthroughs of \cite{vcnew,vcnew0}. We choose to present the result from Theorem 8.4.1 in \cite{anthony1997computational} because it is more concise and explicit. Nevertheless, under our framework, a better bound is indeed obtainable. Second, as shown in the proof, the feasibility result of the theorem holds not just for the optimal solution of SAA, but also for any generic point within the feasible region of SAA. In other words, \cref{main} holds for any algorithm that can output a solution $x^\star(\xi^{[N]})$ (not necessarily the optimal solution of the SAA) in the feasible region of SAA with probability 1. This observation is particularly important when the considered SAA problem is non-convex (i.e., mixed-integer programming), solvable up to local optimum, or requires approximations algorithms.
\end{remark}

There are several advantages when applying \cref{main} to bound the feasibility of SAA solutions: 1) Other than a finite VC dimension, it does not rely on specific assumptions regarding the structures of \eqref{sp} and \eqref{saa}. As we shall see later, even when the chain-constrained domain condition in \cite{liu2019feasibility} becomes restrictive, our analysis based on VC dimension would remain effective.   2) Our bound is explicit and computable with no hidden constants. 3) One might argue the generality of \cref{main} would come at a cost of higher sample complexity. However, as we shall see, this is not the case even when we compare our bounds with some of the best-known ones. Moreover, thanks to 1) above, \cref{main} yields some previously unattainable results. 

While \cref{main} is a result on sample complexity, it is convenient to convert it into an asymptotic rate of convergence with $N$. The infeasibility of SAA solution still decreases exponentially as in \cite{liu2019feasibility}, except the rate is now explicit. We summarize it into a corollary. 
\begin{corollary}
	Under the same conditions of \cref{main}, 
	\begin{equation*}
		\mathbb P^N (V(x^\star(\xi^{[N]}))>\epsilon) \leq 2 \exp \left(-\frac{N\epsilon}{4}\right) \left(\frac{12}{\epsilon}\right)^d.
	\end{equation*}
\end{corollary}

To compare with the existing results in \cite{liu2019feasibility}, we note that direct comparisons on sample complexity are possible only when the rate of convergence is explicit, which only applies to Scenario 1 in \cite{liu2019feasibility} (summarized in Section 2). Moreover, it is shown in \cite{care2014fast} that a relatively tight sufficient condition for $N$ to satisfy
$$\sum_{k=0}^{m-1} {N \choose k} \epsilon^i (1-\epsilon)^{N-i} \leq \delta,$$
is
\begin{equation}\label{scomple}
	N \geq \frac{e}{e-1}\frac{1}{\epsilon}\bigg(m-1+\log\big(\frac{1}{\delta}\big)\bigg),
\end{equation}
which provides a tight bound on sample complexity and we shall make use of it later. 

\section{Applications}
In this section, we apply \cref{main} in several problems of practical interests and compare with established results. For the examples formulated in this section, one can check that $f(\xi,x)$ satisfies the general conditions imposed in the introduction section (below \eqref{sp}). Throughout the proofs, we make use the following lemmas:
\begin{lemma}[Theorem 1.1 from \cite{van2009note}]\label{old}
Given classes of subsets $\mathcal C_1$, $C_2$, ..., $\mathcal C_m$ with $d_{VC}(\mathcal C_j)=d_j<+\infty$, define
    \begin{align*}
\sqcap_{j=1}^m \mathcal C_j\triangleq&\{\cap_{j=1}^m C_j: C_j\in\mathcal C_j, j=1,...,m\}\nonumber\\
\sqcup_{j=1}^m \mathcal C_j\triangleq&\{\cup_{j=1}^m C_j: C_j\in\mathcal C_j, j=1,...,m\},
    \end{align*}
and $d=\sum_{j=1}^m d_j$. Then,
\begin{equation*}
    \max (d_{VC}(\sqcap_{j=1}^m \mathcal C_j), d_{VC}(\sqcup_{j=1}^m \mathcal C_j)) \leq \frac{e}{(e-1)\log2} d \log (\frac{e}{\log 2}m).
\end{equation*}
\end{lemma}

We also use a key result on the upper bound of VC dimension for sets determined by finite-dimensional function spaces. Recall that $\mathcal G$ is a vector space of functions if $a_1f_1+a_2f_2\in\mathcal G$ for any $a_1,a_2\in\mathbb R$ and $f_1,f_2\in\mathcal G$. The lemma below comes directly from \cite{dudley1978central}. For a more concise proof, one is also referred to Lemma 2.6.15 in \cite{van1996weak}.

\begin{lemma}\label{dud}
Given arbitrary space $\mathcal S$, let $\mathcal G$ be a finite-dimensional vector space of functions $g(\cdot): \mathcal S\rightarrow \mathbb R$ with $\text{dim }\mathcal G <+\infty$. Then, the class of sets
\begin{equation*}
    \mathcal H=\{\{s\in\mathcal S: g(s) \geq 0 \}\}_{g\in\mathcal G}
\end{equation*}
has VC dimension at most $\textnormal{dim } \mathcal G$.
\end{lemma}

Note that since $\mathcal G$ is a vector space, $\mathcal H$ can also be characterized as $\{\{s\in\mathcal S: g(s) \leq 0 \}\}_{g\in\mathcal G}$. The following example will be used later.
\begin{example}\label{exdud}
According to \cref{dud}, $\mathcal U=\{\{(y,z)\in\mathbb R^d\times\mathbb R: y^T x \leq z \}\}_{x\in\mathbb R^d}$ has VC dimension at most $d$. In fact, the VC dimension for $\mathcal U$ is equal to $d$ (see  \cite{de2004constraint} or \cite{dudley1978central}).
\end{example}

Finally, related to the reparametrization discussion in \cref{example}, we have the following result.

\begin{lemma}\label{space}
Suppose $\mathcal H=\{H_x\}_{x\in\mathcal X}$ is a collection of subsets on $\Xi$. Given $f(\cdot): \Xi'\rightarrow\Xi$ and $H\subseteq\Xi$, define
\begin{equation*}
    f^{-1}(H)\triangleq\{\xi'\in\Xi', f(\xi')\in H\}
\end{equation*}
with the convention $f^{-1}(\emptyset)=\emptyset$ and define
\begin{equation*}
    \mathcal H'\triangleq \{f^{-1}(H_x)\}_{x\in\mathcal X}.
\end{equation*}
Then, we have $d_{VC}(\mathcal H') \leq d_{VC}(\mathcal H) $.
\end{lemma}
\begin{proof}
    For any set of points $\{\xi'_i\}_{i\in[n]}\subseteq\Xi'$ shattered by $\mathcal H'$, let $\xi_i=f(\xi'_i)$ and consider $\{\xi_i\}_{i\in[n]}$. If we can show $\{\xi_i\}_{i\in[n]}$ is shattered by $\mathcal H$, then we must have $n\leq d_{VC}(\mathcal H)$. Since $n$ is arbitrary, we have:
    \begin{equation*}
        \sup \{n \text{ }| \text{ } H'\subseteq \Xi', |H'|=n \text{ and } H' \text{ is shattered by } \mathcal H' \} \leq d_{VC}(\mathcal H)
    \end{equation*}
  and consequently $d_{VC}(\mathcal H') \leq d_{VC}(\mathcal H)$.
  
  To show $\{\xi\}_{i\in[n]}$ is shattered by $\mathcal H$, pick any $I\subseteq[n]$, and we want to show there exists $x\in\mathcal X$ such that $\{\xi_i\}_{i\in I}\subseteq H_x$ and $\{\xi_i\}_{i\in[n]\cap I^{c}} \cap H_x=\emptyset$. Now, recall $\{\xi'\}_{i\in[n]}$ is shattered by $\mathcal H'$, so there exists $x\in\mathcal X$ such that $\{\xi'_i\}_{i\in I}\subseteq f^{-1}(H_x)$ and $\{\xi'_i\}_{i\in[n]\cap I^{c}} \cap f^{-1}(H_x)=\emptyset$. This concludes the proof, since for such $x\in\mathcal X$, 
  we have 
  \begin{equation*}
      \begin{cases}
		\xi'_i \in f^{-1}(H_x) \implies \xi_i=f(\xi_i')\in H_x , & \text{for } i\in I \\
		\xi'_i \notin f^{-1}(H_x) \implies \xi_i=f(\xi_i')\notin H_x , & \text{for } i\in [n] \cap I^c.
	\end{cases}
  \end{equation*}
\end{proof}

\subsection{Two-Stage Stochastic Programming}

One of the main motivating examples in studying SAA feasibility, mentioned in both \cite{chen2019sample,liu2020asymptotic, liu2019feasibility}, is the two-stage stochastic program without relatively complete recourse (\cite{chen2019sample,liu2020asymptotic, liu2019feasibility}  focus on the linear recourse). This section is divided into three subsections to discuss feasibility in two-stage stochastic programming. The first subsection considers the general case with only continuous variables, the second subsection considers mixed-integer variables, and the third subsection considers low-dimensional structures. Finally, we note that, while we focus on two-stage stochastic programming, these feasibility results can be generalized to multistage stochastic programs using stagewise independence of the data process. For a detailed discussion, see Section 4 of \cite{liu2019feasibility}.

\subsubsection{General Case}
We first consider the following general case. For clarity and generality, we rewrite \eqref{sp} to define $f(\xi,x)$ as follows:
\begin{equation}\label{4-1}
    \inf_{x\in\mathcal X} F(x) \triangleq f_0(x)+\mathbb E [f(\xi,x)],
\end{equation}
with $f_0(x)$ being a deterministic function and
\begin{align}\label{ex1}
\begin{aligned}
		f(\xi,x) \triangleq & \inf_{y}  g(\xi,y) \\
		 \textrm{s.t.} \quad &  W_\xi y+T_\xi x=h_\xi,\\
		& y\geq0,    \\
	\end{aligned}
\end{align}
where we only assume $g(\xi,\cdot)$ is convex and finite everywhere, almost surely $\forall \xi\in\Xi$. Typical classes of problems that fall under this formulation include (see \cite{liu2020asymptotic}):
\begin{itemize}

    \item two-stage stochastic linear programming (SLP): $f_0(x)=c^T x$ is linear, $\mathcal X=\{x\in\mathbb R^n: Ax\leq b\}$ is polyhedral, and $g(\xi,y)=q(\xi)^T y$ is linear in $y$.
    \\
    \item two-stage stochastic quadratic-linear programming (SQLP): 
    
    $f_0(x)=\frac{1}{2}x^TQx+c^T x$ is quadratic in $x$, $\mathcal X=\{x\in\mathbb R^n: Ax\leq b\}$ is polyhedral, and $g(\xi,y)=q(\xi)^T y$ is linear in $y$.
    \\
    \item two-stage stochastic quadratic-quadratic programming (SQQP): 
    
    $f_0(x)= \frac{1}{2}x^TQx+c^Tx $ is quadratic in $x$, $\mathcal X=\{x\in\mathbb R^n: Ax\leq b\}$ is polyhedral, and $g(\xi,y)=\frac{1}{2}y^TP(\xi)y+q(\xi)^T y$ is quadratic in $y$.
\end{itemize}

\text{}

Suppose $g(\xi,y)$ takes the general form above. To derive feasibility results on the SAA solution, \cite{liu2019feasibility} assumes there are only finitely many distinct values for $W_\xi$ or $T_\xi$, i.e., $|\{W_\xi\}|=p$ and $|\{T_\xi\}|=q$ where $\{p,q\}\subseteq \mathbb Z^+$. By Farkas' lemma, $\{y\geq 0: W_\xi y+T_\xi x=h_\xi\} $ is non-empty (i.e., $f(\xi,x)<+\infty$) if and only if $a^T(h_\xi-T_\xi x)\geq 0$ for all $a$ such that $a^T W_\xi \geq 0$. Consequently, as shown in \cite{liu2019feasibility}, suppose we let $W_i, i\in[p]$ to be the $i$-th distinct element in $\{W_\xi\}$, let $\{a_{ij}\}_{j \in J_i}$ to be the set of non-equivalent extreme rays of polyhedral cone $\mathcal C_i=\{a: a^T W_i \geq 0\}$ with $J_i$ being the index set for these extreme rays of $\mathcal C_i$, and let $I(\cdot):\Xi\rightarrow[p]$ be the indexing function such that $I(\xi)=i$ when $W_\xi=W_i$. Then, we have
\begin{equation}\label{dom}
	\text {dom } f_\xi=\{x\in\mathcal X \text{ }| \text{ } a_{I(\xi)j}^T T_\xi x \leq a_{I(\xi)j}^T h_\xi, \forall j \in J_{I(\xi)}\},
\end{equation}
where we assume the cone $\mathcal C_i$ are pointed, so they can be minimally generated by a unique set of non-equivalent extreme rays. Otherwise, a general polyhedral cone is still finitely generated according to Weyl-Minkowski’s Theorem, but the minimal set of generators may not be unique. Without of loss of generality (as in most linear programming problems), we assume the rows of $W_i$ are linearly independent, which implies the cone $\mathcal C_i$ is pointed (see, e.g., \cite{terzer2009large}). Now, note that \eqref{dom} allows \cite{liu2019feasibility} to use the so-called ``chain-constrained domain" structure. Here, a chain-constrained domain is defined as follows:
\begin{definition}\label{chain}
	A collection of functions $\{f(\xi,\cdot)\}_{\xi\in\Xi}$ has chain-constrained domain of order $m$ if there exist $m$ chains $\{U_k^\xi\}_{\xi\in\Xi}$ and 
	\begin{equation*}
		\text{ dom } f_\xi = \bigcap_{k=1}^m U_k^\xi
	\end{equation*} 
	where a collection of sets $\{U^\omega\}_{\omega \in I}$ is a chain if for any $\omega_1, \omega_2 \in I$, we have either $U_{\omega_1} \subseteq U_{\omega_2} $ or $U_{\omega_2} \subseteq U_{\omega_1} $.
\end{definition}

It is shown in \cite{liu2019feasibility} that $\text{dom } f_\xi$ in \eqref{dom} is a chain-constrained domain of order $m=q\sum_{i=1}^p |J_i|$. Consequently, Scenario 1 in \cite{liu2019feasibility} can be applied to show that
\begin{equation*}
	\mathbb P^N (V(x^\star (\xi^{[N]})) > \epsilon) \leq \sum_{k=0}^{m-1}  {N \choose k} \epsilon^k(1-\epsilon)^{N-k},
\end{equation*}
which has a sample complexity 
\begin{equation}\label{s1}
	\frac{e}{e-1}\frac{1}{\epsilon}\big(m-1+\log\big(\frac{1}{\delta}\big)\big)
\end{equation}
for achieving $\mathbb P^N (V(x^\star (\xi^{[N]})) > \epsilon) \leq \delta$ according to \eqref{scomple}.

Notice a necessary assumption made in \cite{liu2019feasibility} is that only finitely many distinct values for $W_\xi$ or $T_\xi$ are allowed. However, using \cref{main}, we can get a different sample complexity and concentration bounds, even when the cardinalities of $|\{W_\xi\}|$ and $|\{T_\xi\}|$ are infinite. In particular, we first show feasibility results on \eqref{ex1} when $|\{W_\xi\}|$ and $|\{T_\xi\}|$ are finite in \cref{2-stage}. Then, we extend the result to the case when $|\{W_\xi\}|$ and $|\{T_\xi\}|$ are infinite in \cref{3-stage}.
\begin{corollary}\label{2-stage}
Consider \eqref{ex1} and assume there are only finitely many distinct values for $W_\xi$ or $T_\xi$ , i.e., $|\{W_\xi\}|=p$ and $|\{T_\xi\}|=q$ where $\{p,q\}\subseteq \mathbb Z^+$. Let $\xi^{[N]}=\{\xi_1,...,\xi_N\}$ be IID samples from $\mathbb P$ (consequently $\xi^{[N]}\sim \mathbb P^N$), and $x^\star (\xi^{[N]})$ be the output of any algorithm that is guaranteed to be within the feasible region of SAA. Then, if 
\begin{equation}\label{s2}
	 N\geq \frac{4}{\epsilon}\bigg(d_{VC}\log\big(\frac{12}{\epsilon}\big)+\log\big(\frac{2}{\delta}\big)\bigg),
\end{equation}
where $|J|=\max_{i\in[p]}|J_i|$ and 
\begin{equation}\label{dvc}
    d_{VC}=\Big(\frac{e}{(e-1)\log2} |J|(n+1) \log \big(\frac{e}{\log2}\cdot |J|\big)\Big),
\end{equation}
we have
	\begin{equation*}
\mathbb P^N(V(x^\star (\xi^{[N]}))>\epsilon) \leq  \delta.
	\end{equation*}
for any $0<\delta,\epsilon <1$. Equivalently, in terms of convergence rate, we have
\begin{equation}
    \mathbb P^N (V(x^\star(\xi^{[N]}))>\epsilon) \leq 2 \exp \left(-\frac{N\epsilon}{4}\right) \left(\frac{12}{\epsilon}\right)^{d_{VC}}.
\end{equation}
\end{corollary}
\begin{proof}
Define $I(\cdot):\Xi\rightarrow[p]$ as the indexing function such that $I(\xi)=i$ when $W_\xi=W_i$. Then, it follows that $\mathcal H \triangleq \{H_x\}_{x\in\mathcal X}\cup\{\Xi\}$ defined in \eqref{domm} consists of $$H_x=\{\xi: a_{I(\xi) j}^T T_\xi x \leq a_{I(\xi) j}^T h_\xi, \forall j\in J_{I(\xi)}\}$$ where $\{a_{ij}\}_{j \in J_i}$ is the set of non-equivalent extreme rays of polyhedral cone $\{a: a^T W_i \geq 0\}$. Define $|J|=\max_{i\in[q]}|J_i|$ and, given $\xi\in\Xi$, define $\{(y_{\xi j},z_{\xi j})\}_{j\in |J|}\subseteq \mathbb R^{n}\times \mathbb R$ to be:

\begin{equation*}
	y^T_{\xi j}\triangleq
	\begin{cases}
		a_{I(\xi) j}^T T_\xi , & \text{for } 1 \leq j \leq |J_{I(\xi)}| \\
		\boldsymbol{0}, & \text{for } |J_{I(\xi)}| < j \leq |J| 
	\end{cases}
\end{equation*}

\begin{equation*}
	z_{\xi j}\triangleq
	\begin{cases}
		a_{I(\xi) j}^T h_\xi , & \text{for } 1 \leq j \leq |J_{I(\xi)}| \\
		{0}, & \text{for } |J_{I(\xi)}| < j \leq |J|.
	\end{cases}
\end{equation*}
Then, define $y_\xi^T=(y^T_{\xi 1}, y^T_{\xi 2},..., y^T_{\xi |J|})\in\mathbb R^{|J|n}$ and $z_\xi=(z_{\xi 1}, z_{\xi 2},..., z_{\xi |J|})\in\mathbb R^{|J|}$. Moreover, for $j\in[|J|]$, define $v_j(\cdot):\mathbb R^{n}\rightarrow\mathbb R^{|J|n}$ to be 
\begin{equation*}
	[v_{j}(x)]_i= 
	\begin{cases}
	[x]_i, & \text{ for  $(j-1)n+1\leq i \leq jn$ } \\
		{0}, & \text{otherwise}. 
	\end{cases}
\end{equation*}
Then, we can redefine 
\begin{equation}\label{use}
    H_x=\bigcap_{j=1}^{|J|} \{(y_\xi,z_\xi): y_{\xi}^T v_j(x) \leq [z_{\xi}]_j\}.
\end{equation}
Fix $j\in[|J|]$, let ${e}_j\in\mathbb R^{|J|}$ be the vector with 1 in the $j$-th component and 0 otherwise. Define a class of function $\mathcal G_j=\{g_{(x,c)}(\cdot)\}_{(x,c)\in\mathbb R^n\times \mathbb R}$ on $\mathbb R^{|J|(n+1)}$ such that, given $(y,z)\in\mathbb R^{|J|n}\times \mathbb R^{|J|}$,
\begin{equation*}
g_{(x,c)}((y,z))=[y,z]^T  \begin{bmatrix}
           -v_j(x) \\
           c\cdot e_j
\end{bmatrix}.
\end{equation*}
It is straightforward to check $\mathcal G_j$ is a finite-dimensional vector space of functions with $\text{dim } \mathcal G_j \leq n+1$.  Then, according to \cref{dud}, the VC dimension of $$\{\{(y,z)\in\mathbb R^{|J|n}\times \mathbb R^{|J|}: g_{x,c}((y,z))\geq 0\}\}_{(x,c)\in\mathbb R^n\times\mathbb R}$$
is at most $n+1$. Moreover, by letting $(x,c)=\boldsymbol{0}$, it can be seen that the above collection of sets includes the set $\mathbb R^{|J|n}\times \mathbb R^{|J|}$. Consequently, as a smaller collection of sets, the VC dimension of $$\{\{(y,z)\in\mathbb R^{|J|n}\times \mathbb R^{|J|}: g_{x,1}((y,z))\geq 0\}\}_{x\in\mathcal X} \cup \{\mathbb R^{|J|n}\times \mathbb R^{|J|}\} $$
is at most $n+1$. Thus, for each $j\in [|J|]$, it follows from \cref{space} that the VC dimension of $$\mathcal U_j=\{\{(y_\xi,z_\xi): y_{\xi}^T v_j(x) \leq [z_{\xi}]_j\}\}_{x\in\mathcal X}\cup \{\Xi\}$$ is at most $n+1$. Finally, it follows from \cref{old} and \eqref{use} that 
\begin{equation*}
	d_{VC}(\mathcal H) \leq d_{VC}(\sqcap_{j=1}^{|J|} \mathcal U_j) \leq \frac{e}{(e-1)\log2} |J|(n+1) \log \left(\frac{e}{\log2}\cdot |J|\right).
\end{equation*} 
The corresponding sample complexity and convergence rate follow from \cref{main}.
\end{proof}

Note that \cref{2-stage} does not require convexity assumption on $g_\xi$ or distributional assumptions on the random variables $W_\xi$ and $T_\xi$. In fact, we can further extend our result to the case when  $|\{W_\xi\}|$ and $|\{T_\xi\}|$ are infinite. In particular, the same proof can be applied as long as $$|J|=\max_{\xi\in\Xi} \{\text{ \# of extreme rays for the cone  }\{a: a^T W_\xi \geq 0\}\}$$ is finite. However, it is known that the number of non-equivalent extreme rays of a polyhedral cone $\{a:a^T W\geq 0\}$ is finite and can be bounded by a term of ${n_1 \choose m_1-1}$, which only involves parameters $m_1$ and $n_1$ for $W\in\mathbb R^{m_1\times n_1}$ (similar to the bound on the number of extreme points, $m_1-1$ linearly independent ``constraints" are chosen; for details see \cite{ jing2020complexity,Panik1993,terzer2009large}). Thus, as long as $\{W_\xi\}\subseteq \mathbb R^{m_1\times n_1}$ and $m_1, n_1$ are bounded almost surely, we have $|J|<+\infty$ regardless of the cardinalities of $\{W_\xi\}$. We summarize this into another corollary. 

\begin{corollary}\label{3-stage}

Consider \eqref{ex1} and assume $|J|<\infty$ where 
\begin{equation*}
    |J|=\max_{\xi\in\Xi} \{\text{ \# of extreme rays for the cone  }\{a: a^T W_\xi \geq 0\}\}.
\end{equation*}
Let $\xi^{[N]}=\{\xi_1,...,\xi_N\}$ be IID samples from $\mathbb P$ (consequently $\xi^{[N]}\sim \mathbb P^N$), and $x^\star (\xi^{[N]})$ be the output of any algorithm that is guaranteed to be within the feasible region of SAA. Then the result of \cref{2-stage} still holds. 
\end{corollary}

	
\begin{proof}
 Let $\mathcal A_\xi$ be the set of non-equivalent extreme rays of polyhedral cone $\{a: a^T W_\xi \geq 0\}$. Observe $\mathcal H \triangleq \{H_x\}_{x\in\mathcal X}\cup \{\Xi\}$ defined in \eqref{domm} consists of
\begin{equation}\label{uselater}
    H_x=\{\xi: a^T T_\xi x \leq a^T h_\xi, \forall a\in \mathcal A_\xi\}.
\end{equation}
 For all $\xi\in\Xi$, since $|\mathcal A_\xi|\leq |J|<+\infty$, we can label the elements in $\mathcal A_\xi$ by $\{a_{\xi j}\}_{j\in\mathcal [|A_\xi|]}$. Then, define $\{(y_{\xi j},z_{\xi j})\}_{j\in |J|}$ as
\begin{equation*}
	y^T_{\xi j}= 
	\begin{cases}
		a_{\xi j}^T T_\xi , & \text{for } 1 \leq j \leq |\mathcal A_{\xi}| \\
		\boldsymbol{0}, & \text{for } |\mathcal A_{\xi}| < j \leq |J| 
	\end{cases}
\end{equation*}

\begin{equation*}
	z_{\xi j}= 
	\begin{cases}
		a_{\xi j}^T h_\xi , & \text{for } 1 \leq j \leq |\mathcal A_{\xi}| \\
		{0}, & \text{for } |\mathcal A_{\xi}| < j \leq |J|.
	\end{cases}
\end{equation*}
The rest of proof follows exactly as in Corollary \ref{2-stage}.
\end{proof}

Compared with our bound \eqref{s2}, the chain-constrained bound \eqref{s1} relies on the order of the chain $m=q\sum_{i=1}^p |J_i|$. If the cardinality of $\{W_\xi\}$ or $\{T_\xi\}$ gets large (i.e., $qp\gg n$), or potentially infinite (for continuous random variable), then the bound in \eqref{s1} with a sample complexity of $ O(\frac{qp |J|}{\epsilon}+\frac{1}{\epsilon} \log(\frac{1}{\delta}))$ becomes loose or even inapplicable. On the other hand, the VC bound \eqref{s2} with a sample complexity $O(\frac{|J|n}{\epsilon}\log |J| \log (\frac{1}{\epsilon})+\frac{1}{\epsilon} \log(\frac{1}{\delta}))$ maintains the same dependence on the dimension $n$ regardless of the support of $W_\xi$ or $T_\xi$. Moreover, if we use the PAC bound from \cite{vcnew,vcnew0} as mentioned in \cref{t1r}, the bound would be improved to $O(\frac{|J|n}{\epsilon}\log |J|+\frac{1}{\epsilon} \log(\frac{1}{\delta}))$. Finally, in both bounds, the term $|J|$ appears. However, as mentioned previously, an explicit bound for $|J|$ in terms of $m_1,n_1$ can be obtained, where $\{W_\xi\}\subseteq \mathbb R^{m_1 \times n_1}$. We omit it here as it is not essential for our comparison. Finally, the bound in Scenario 3 of \cite{liu2019feasibility} also applies to \eqref{ex1} and is not limited by the order of the chain-structure. However, the bound there is not explicitly computable due to the hidden term $\beta $.

The dependence on the order of the chain $m$ is also discussed in \cite{chen2019sample}. Using ideas similar to the scenario approximation of chance-constrained problems in \cite{calafiore2005uncertain,luedtke2008sample,campi2008exact}, as well as specific properties of linear programming (e.g., existence of basic optimal solutions), \cite{chen2019sample} is able to provide a sample complexity for two-stage stochastic linear programming independent of the cardinalities of $\{W_\xi\}$ or the order of the chain.  Nonetheless, the derivation of our bound in \eqref{s2} does not depend on the linearity of the optimization problem and hence is not limited to two-stage stochastic programming with linear recourse. More specifically, in \cite{chen2019sample}, the first stage $\mathcal X$ is defined by linear constraints $Ax=b$ for some $A\in\mathbb R^{m\times n}$ and the second stage problem bears a linear objective $q(\xi)^Ty$. In contrast, our bound is valid for general $\mathcal X$ in the first stage and $g(\xi,y)$ in the second-stage problem in \eqref{ex1}. That being said, the bound derived in \cite{chen2019sample} has notable strengths in the linear case, in terms of the dependence on problem parameters, gained via a more efficient exploitation of the linear structure. Specifically, the sample complexity in \cite{chen2019sample} is (adapted to the notation in this paper)
\begin{equation}\label{cheny}
    O\bigg(\frac{1}{\epsilon}\Big(nn_1\big(\log(\frac{m_1}{n_1+1})+1\big)+n\big(\log(\frac{m}{n}+2)+\log(\frac{1}{\epsilon})+1\big)+\log(\frac{1}{\delta})\Big)\bigg),
\end{equation}
which has better dependence on $m_1,n_1$, as the dependence on $|J|$ in \eqref{s2} is ${n_1 \choose m_1-1}$ in the worst case. Nonetheless, \eqref{s2} has a similar dependence on $n$ as the bound in \eqref{cheny}, and does not depend on $m$ in \eqref{cheny} at all. Omitting the dependence on these problem size parameters (e.g., constants based on $n,m, m_1, n_1$ and $|J|$), the bound derived in \cite{chen2019sample} is of order $O(\frac{1}{\epsilon}\log(\frac{1}{\delta})+\frac{1}{\epsilon}\log(\frac{1}{\epsilon}))$, which of the same order as the bound \eqref{s2}. Moreover, \eqref{s2} can be slightly improved to be of order  $O(\frac{1}{\epsilon}\log(\frac{1}{\delta})+\frac{1}{\epsilon})$ bound based on \cref{t1r}.

\subsubsection{Mixed-Integer Programming}
The SAA method has also been applied in two-stage stochastic programming with mixed-integer recourse \cite{ahmed2010two, ahmed2002sample,kuccukyavuz2017introduction,bidhandi2017accelerated}. However, due to the presence of integer variable, currently known results based on \eqref{4-1} regarding the feasibility of SAA solutions do not apply. In this section, we provide an original feasibility bound in this case. We consider the following two-stage stochastic mixed-integer programming where $\mathcal X\subseteq \mathbb R^{n-p}\times \mathbb Z^p$ is allowed to contain integer components in the first stage
\begin{equation*}
    \inf_{x\in\mathcal X} F(x) \triangleq f_0(x)+\mathbb E [f(\xi,x)],
\end{equation*}
and the second stage is a mixed-integer program (MIP): 
\begin{equation}\label{ex10}
	\begin{aligned}
		f(\xi,x) \triangleq  \inf_{y} \quad & g(\xi,y,y_0) \\
		 \textrm{s.t.} \quad&  W_\xi y+W^0_\xi y_0+T_\xi x=h_\xi,\\
		& y\in\mathbb R^{n'}_{+}, y_0\in \mathcal Z\subseteq \mathbb Z_{+}^{p'},\\
	\end{aligned}
\end{equation}
for given $n',p'\in\mathbb Z_{+}$. Here $g(\xi,y,y_0)$ can be a general function as in \eqref{ex1}, although for much of theoretical and practical interest (also applicability), it is assumed to be in linear form $g(\xi,y,y_0)=q(\xi)^Ty+q_0(\xi)^Ty_0$. Moreover, most literature also assumes relatively complete recourse by fixing a deterministic recourse matrix (i.e., $W_\xi=W$ and $W_\xi^0=W^0$ with probability 1) such that $\{\boldsymbol{v}\in\mathbb R^{n'}_{+}\times \mathbb Z_{+}^{p'}: [W|W^0]\boldsymbol{v}=\boldsymbol{w}\}$ is non-empty for all relevant $\boldsymbol{w}$. Consequently, the feasibility of SAA solution for two-stage stochastic integer programming without relatively complete recourse has rarely been considered. In fact, due to the general non-convex and discontinuous nature of MIP, specialized approximation or iterative algorithms are usually required and the solutions are no longer guaranteed to be optimal. However, even without relatively complete recourse or optimality guarantee, as mentioned in \cref{t1r}, as long as the solutions output from such algorithms are within the SAA feasible region with probability 1, the feasibility result from \cref{main} still holds. Recall we have assumed the set $\{x: x\in\mathcal X \text{ and } F(x)<+\infty\}$  is non-empty and the SAA feasible region is non-empty with probability 1. 

Under the setting of \cref{main}, it is possible to provide a feasibility bound for \eqref{ex10} when $|\mathcal Z|<+\infty$. This condition is satisfied when $y_0$ is restricted to be binary as in \cite{bidhandi2017accelerated} (i.e., $y_0\in{\{0,1\}}^{p'}$). On the other hand, if the solutions are polynomially bounded by the size of data (e.g., integer linear programming \cite{borosh1976bounds}), then it is also possible to consider solving \eqref{ex10} in a finite, although possibly large bounded set $\mathcal Z \subseteq \mathbb Z^{p'}_+$ thus satisfying $|\mathcal Z|<+\infty$.

\begin{corollary}\label{int}
Consider \eqref{ex10}. Suppose $|\mathcal Z|<+\infty$ and $|J|<+\infty$ where $$|J|=\max_{\xi\in\Xi} \{\text{ \# of extreme rays for the cone  }\{a: a^T W_\xi \geq 0\}\}.$$ Then, let $\xi^{[N]}=\{\xi_1,...,\xi_N\}$ be IID samples from $\mathbb P$ (consequently $\xi^{[N]}\sim \mathbb P^N$), and $x^\star (\xi^{[N]})$ be the output of any algorithm that is guaranteed to be within the feasible region of SAA. Then, if 
\begin{equation}
	 N\geq \frac{4}{\epsilon}\bigg(d_{VC}\log\big(\frac{12}{\epsilon}\big)+\log\big(\frac{2}{\delta}\big)\bigg),
\end{equation}
where 
\begin{equation*}
d_{VC}=\Big(\frac{e}{(e-1)\log2}\Big)^2|\mathcal Z| |J|(n+2) \log (\frac{e|J|}{\log 2})\log (\frac{e|\mathcal Z|}{\log 2}),
\end{equation*}
then we have $\mathbb P^N(V(x^\star (\xi^{[N]}))>\epsilon) \leq  \delta$, for any $0<\delta,\epsilon <1$. Equivalently, in terms of convergence rate, we have
\begin{equation}
    \mathbb P^N (V(x^\star(\xi^{[N]}))>\epsilon) \leq 2 \exp \left(-\frac{N\epsilon}{4}\right) \left(\frac{12}{\epsilon}\right)^{d_{VC}}.
\end{equation}
\end{corollary}
\begin{proof}
    Let $\mathcal A_\xi$ be the set of non-equivalent extreme rays of polyhedral cone $\{a: a^T W_\xi \geq 0\}$.
    Using Farkas' lemma as in \eqref{uselater}, we construct elements of $\mathcal H \triangleq \{H_x\}_{x\in\mathcal X}\cup\{\Xi\}$ defined in \eqref{domm} as
    \begin{equation*}
      H_x=\bigcup_{y_0\in\mathcal Z} \{\xi: a^T (T_\xi x+ W_\xi^0 y_0) \leq a^T h_\xi, \forall a\in \mathcal A_\xi\}, 
    \end{equation*}
     for $(x,y_0)\in\mathcal X \times\mathcal Z$. For all $\xi\in\Xi$, since $|\mathcal A_\xi|\leq |J|<+\infty$, we can label the elements in $\mathcal A_\xi$ by $\{a_{\xi j}\}_{j\in\mathcal [|A_\xi|]}$. Then, given $\xi\in\Xi$, define $\{(y_{\xi j},z_{\xi j},w_{\xi j})\}_{j\in |J|}\subseteq \mathbb R^{n}\times \mathbb R \times \mathbb R^{p'}$ as
        \begin{equation*}
	y^T_{\xi j}= 
	\begin{cases}
		a_{\xi j}^T T_\xi , & \text{for } 1 \leq j \leq |\mathcal A_{\xi}| \\
		\boldsymbol{0}, & \text{for } |\mathcal A_{\xi}| < j \leq |J| 
	\end{cases}
\end{equation*}

\begin{equation*}
	z_{\xi j}= 
	\begin{cases}
		a_{\xi j}^T h_\xi , & \text{for } 1 \leq j \leq |\mathcal A_{\xi}| \\
		{0}, & \text{for } |\mathcal A_{\xi}| < j \leq |J| 
	\end{cases}
\end{equation*}

\begin{equation*}
	w^T_{\xi j}= 
	\begin{cases}
		a_{\xi j}^T W^0_\xi , & \text{for } 1 \leq j \leq |\mathcal A_{\xi}| \\
		\boldsymbol{0}, & \text{for } |\mathcal A_{\xi}| < j \leq |J|. 
	\end{cases}
\end{equation*}
Define $y^T_\xi=(y^T_{\xi 1}, y^T_{\xi 2},..., y^T_{\xi |J|})\in\mathbb R^{|J|n}$, $z_\xi=(z_{\xi 1}, z_{\xi 2},..., z_{\xi |J|})\in\mathbb R^{|J|}$ and $w^T_\xi=(w^T_{\xi 1}, w^T_{\xi 2},..., w^T_{\xi |J|})\in\mathbb R^{|J|p'}$. Moreover, for $j\in[|J|]$, define $v_j(\cdot):\mathbb R^{n}\rightarrow\mathbb R^{|J|n}$, $u_j: \mathbb Z^{p'}\rightarrow\mathbb Z^{|J|p'}$ so that
\begin{equation*}
	[v_{j}(x)]_i= 
	\begin{cases}
	[x]_i, & \text{ for  $(j-1)n+1\leq i \leq jn$ } \\
		{0}, & \text{otherwise} 
	\end{cases}
\end{equation*}
\begin{equation*}
	[u_{j}(x)]_i= 
	\begin{cases}
	[y_0]_i, & \text{ for  $(j-1)p'+1\leq i \leq jp'$ } \\
		{0}, & \text{otherwise}.
	\end{cases}
\end{equation*}
Then, we can redefine 
\begin{equation}\label{intH}
    H_x=\bigcup_{y_0\in\mathcal Z}\bigcap_{j=1}^{|J|} \{(y_\xi,z_\xi,w_\xi): y_{\xi}^T v_j(x)+ w_\xi^T u_j(y_0) \leq [z_{\xi}]_j\}.
\end{equation}
Given $j\in[|J|]$ and $y_0\in\mathcal Z$, let ${e}_j\in\mathbb R^{|J|}$ be the vector with 1 in the $j$-th component and 0 otherwise. Define a class of function $\mathcal G=\{g_{(x,c_1,c_2)}(\cdot)\}_{(x,c_1,c_2)\in\mathbb R^n\times\mathbb R\times\mathbb R}$ on $\mathbb R^{|J|(n+1+p')}$ such that, given $(y,z,w)\in\mathbb R^{|J|n}\times \mathbb R^{|J|}\times \mathbb R^{|J|p'}$,
\begin{equation*}
g_{(x,c_1,c_2)}((y,z,w))=[y,z,w]^T  \begin{bmatrix}
           -v_j(x) \\
           c_1\cdot e_j\\
           -c_2\cdot u_j(y_0)
\end{bmatrix}.
\end{equation*}
It is straightforward to check $\mathcal G$ is a finite-dimensional vector space of functions with $\text{dim } \mathcal G \leq n+2$.  Then, according to \cref{dud}, the VC dimension of $$\{\{(y,z,w)\in\mathbb R^{|J|n}\times \mathbb R^{|J|}\times\mathbb R^{|J|p'}: g_{(x,c_1,c_2)}((y,z,w))\geq 0\}\}_{(x,c_1,c_2)\in\mathbb R^n\times\mathbb R\times\mathbb R}$$
is at most $n+2$. Moreover, by letting $(x,c_1,c_2)=\boldsymbol{0}$, it can be seen that the above collection of sets include the set $\mathbb R^{|J|n}\times \mathbb R^{|J|}\times \mathbb R^{|J|p'}$. Consequently, as a smaller collection of sets, the VC dimension of $$\{\{(y,z,w)\in\mathbb R^{|J|n}\times \mathbb R^{|J|}\times\mathbb R^{|J|p'}: g_{(x,1,1)}((y,z,w))\geq 0\}\}_{x\in\mathcal X} \cup \{\mathbb R^{|J|n}\times \mathbb R^{|J|}\times \mathbb R^{|J|p'}\}$$
is also at most $n+2$. Thus, for each $j\in [|J|]$, it follows from \cref{space} that the VC dimension of $$\mathcal U_j^{y_0}=\{\{(y_\xi,z_\xi, w_\xi): y_{\xi}^T v_j(x)+ w_\xi^T u_j(y_0) \leq [z_{\xi}]_j\}\}_{x\in\mathcal X} \cup \{\Xi\}$$ is at most $n+2$. Consequently, given $y_0\in\mathcal Z$, it follows from \cref{old} that 
\begin{align*}
    d_{VC}(\sqcap_{j=1}^{|J|}\mathcal U_j^{y_0}) \leq \frac{e}{(e-1)\log2} |J|(n+2) \log (\frac{e}{\log 2}|J|)
\end{align*}
and then
\begin{align*}
    d_{VC}\Big(\sqcup_{y_0\in\mathcal Z}\Big(\sqcap_{j=1}^{|J|}\mathcal U_j^{y_0}\Big)\Big)\leq \Big(\frac{e}{(e-1)\log2}\Big)^2|\mathcal Z| |J|(n+2) \log (\frac{e|J|}{\log 2})\log (\frac{e|\mathcal Z|}{\log 2}).
\end{align*}
Thus, for $\mathcal H$ defined in \eqref{intH}, we have 
$$d_{VC}(\mathcal H)\leq\Big(\frac{e}{(e-1)\log2}\Big)^2|\mathcal Z| |J|(n+2) \log (\frac{e|J|}{\log 2})\log (\frac{e|\mathcal Z|}{\log 2}).$$ The rest of the proof follows as in \cref{2-stage}.
\end{proof}
As we can see, the portion of infeasible SAA solutions (not necessarily optimal) still decreases exponentially as the sample size $N$ increases, although it is worth noting that the rate now depends on $|\mathcal Z|$ as well.

\subsubsection{Low-Dimensional Models}
There has been a growing literature of applying and analyzing SAA in high-dimensional stochastic programming by leveraging low-dimensional structures, e.g., sparsity, low-rankness \cite{bugg2021logarithmic, liu2019regularized, liu2019sample}. Most of these results focus on the optimization accuracy of SAA as in \cite{shapiro2005complexity}, which compute a sample complexity of 
\begin{equation}\label{scomp}
    N=O\bigg(\frac{n}{\epsilon^2}\log\frac{1}{\epsilon}+\frac{1}{\epsilon^2}\log\frac{1}{\delta}\bigg)
\end{equation}
to guarantee 
\begin{equation}\label{optimal}
    \mathbb P^{N}\bigg(F(x^\star(\xi^{[N]})-\inf_{x\in\mathcal X} F(x) > \epsilon \bigg) \leq \delta.
\end{equation}
Typical results exploring low-dimensional structure aim to reduce the dependence of \eqref{scomp} on $n$, especially for $n\gg\frac{1}{\epsilon}$. For example, the main result on sample complexity in \cite{liu2019sample} trades a worse dependence on $\epsilon$ for a better dependence on $n$ in high-dimensional problem to obtain
\begin{equation}\label{newsp}
  N=O\bigg(\frac{n_0}{\epsilon^3}\Big(\log\frac{n}{\epsilon}\Big)^{1.5}+\frac{1}{\epsilon^2}\log\frac{1}{\delta}\bigg),
\end{equation}
through certain sparse modifications of SAA, where $n_0$ is the number of non-zero entries for the optimal solution of \eqref{sp} with $n_0\ll n$. Compared to \eqref{scomp}, the sample complexity \eqref{newsp} depends polynomially on $n_0$ and $\log n$, instead of $n$ which could become prohibitively large in high-dimensional problems. 

As we shall see, this type of trade-off in sample complexity is typical and it also applies to the feasibility guarantee. In this section, we illustrate how low-dimensional modeling assumptions can allow alternative feasibility bounds for two-stage stochastic programming, especially in high-dimensional settings similar to \eqref{newsp}. In particular, we still consider \eqref{ex1}, but focus on the setting where the optimal solution of \eqref{sp} (i.e., $\text{argmin } F(x)$) and the solution output of SAA (i.e., $x^\star(\xi^{[N]})$) are both sparse:
\begin{equation*}
    \mathcal X\subseteq \{x\in\mathbb R^n: \|x\|_0 \leq n_0\}
\end{equation*}
with $n_0\ll n$. As before, we do not require $x^\star(\xi^{[N]})$ to be optimal for SAA, but only to be within the feasible domain of SAA with probability 1.

The sparsity of the solution is usually achieved by regularization penalty which includes convex penalty, i.e., $\ell_1$-norm (LASSO), $\ell_2$-norm (ridge) or $\ell_p$-norm ($1\leq p\leq+\infty$), and nonconvex penalty, i.e., folded concave penalty  (see, e.g., \cite{liu2019sample}) or $\ell_0$-norm (subset selection with mixed-integer programming; see, e.g., \cite{bertsimas2016best}). While the level of sparsity of $x^\star(\xi^{[N]})$ can be controlled by hyperparameters, we do not generally know the true value of $\| \text{argmin } F(x)\|_0$. In practice, the choice of $n_0$ is chosen either by model selection techniques or fixed beforehand based on domain-specific knowledge. To simplify the problem, we assume the choice of $n_0$, the level of sparsity for candidates solutions, is pre-set. Ideally, we have $\| \text{argmin } F(x)\|_0\leq n_0 \ll n$ so that the space of candidate solutions is rich enough to include the true sparse solution, but still highly sparse compared to problem dimension $n$. To provide our alternative feasibility bound, we first present the following lemma on the VC dimensions of sparse linear classifiers.

\begin{lemma}\label{sparse}
    Let $\mathcal U=\{\{(y,z)\in\mathbb R^d\times\mathbb R: y^T x \leq z \}\}_{x\in\mathbb R^d, \|x\|_0\leq d_0}$. Then $$d_{VC}( \mathcal U \cup \{\mathbb R^d \times \mathbb R\})\leq 2(d_0+1)\log_2\Big(\frac{de+e}{d_0+1}\Big).$$
\end{lemma}
\begin{proof}
    It follows from Lemma 1 in \cite{abramovich2018high} that the VC dimension of $$\mathcal U'\triangleq \{\{(y,z)\in\mathbb R^d\times\mathbb R: y^Tx+z\cdot c\leq 0\}\}_{(x,c)\in\mathbb R^d\times \mathbb R, \|x\|_0+\|c\|_0\leq d_0+1},$$
satisfies $d_{VC}(\mathcal U') \leq 2(d_0+1)\log_2\Big(\frac{de+e}{d_0+1}\Big)$. Now notice $\mathcal U \subseteq \mathcal U'$, since we can rewrite $$\mathcal U= \{\{(y,z)\in\mathbb R^d\times\mathbb R: y^Tx+z\cdot c\leq 0\}\}_{x\in\mathbb R^d, c=-1, \|x\|_0\leq d_0}.$$
On the other hand, $\mathbb R^d\times \mathbb R\in \mathcal U'$ by picking $(x,c)=\boldsymbol 0$. Consequently, $$d_{VC}(\mathcal U \cup \{\mathbb R^d\times\mathbb R\}) \leq d_{VC}(\mathcal U') \leq 2(d_0+1)\log_2\Big(\frac{de+e}{d_0+1}\Big).$$
\text{ }
\end{proof}

Now we present the alternative version of \cref{3-stage} under sparsity.
\begin{corollary}\label{spaa}
Consider \eqref{ex1}. Suppose $\mathcal X\subseteq \{x\in\mathbb R^n: \|x\|_0 \leq n_0\}$ and $|J|<+\infty$ where $$|J|=\max_{\xi\in\Xi} \{\text{ \# of extreme rays for the cone  }\{a: a^T W_\xi \geq 0\}\}.$$ Then, let $\xi^{[N]}=\{\xi_1,...,\xi_N\}$ be IID samples from $\mathbb P$ (consequently $\xi^{[N]}\sim \mathbb P^N$), and $x^\star (\xi^{[N]})$ be the output of any algorithm that is guaranteed to be within the feasible region of SAA. Then, if 
\begin{equation}
	 N\geq \frac{4}{\epsilon}\bigg(d_{VC}\log\big(\frac{12}{\epsilon}\big)+\log\big(\frac{2}{\delta}\big)\bigg),
\end{equation}
where 
\begin{equation*}
d_{VC}=\bigg(\frac{e}{(e-1)\log2} |J|\Big(2(n_0+1)\log_2\Big(\frac{ne+e}{n_0+1}\Big)\Big)\cdot \log \big(\frac{e}{\log2}\cdot |J|\big)\bigg),
\end{equation*}
then we have $\mathbb P^N(V(x^\star (\xi^{[N]}))>\epsilon) \leq  \delta$, for any $0<\delta,\epsilon <1$. Equivalently, in terms of convergence rate, we have
\begin{equation}
    \mathbb P^N (V(x^\star(\xi^{[N]}))>\epsilon) \leq 2 \exp \left(-\frac{N\epsilon}{4}\right) \left(\frac{12}{\epsilon}\right)^{d_{VC}}.
\end{equation}
\end{corollary}
\begin{proof}
    As in \cref{3-stage}, let $\mathcal A_\xi$ be the set of non-equivalent extreme rays of polyhedral cone $\{a: a^T W_\xi \geq 0\}$. Observe $\mathcal H \triangleq \{H_x\}_{x\in\mathcal X}\cup \{\Xi\}$ defined in \eqref{domm} consists of
\begin{equation*}
    H_x=\{\xi: a^T T_\xi x \leq a^T h_\xi, \forall a\in \mathcal A_\xi\}.
\end{equation*}

For all $\xi\in\Xi$, since $|\mathcal A_\xi|\leq |J|<+\infty$, we can label the elements in $\mathcal A_\xi$ by $\{a_{\xi j}\}_{j\in\mathcal [|A_\xi|]}$. Then, define $\{(y_{\xi j},z_{\xi j})\}_{j\in |J|}$ as

\begin{equation*}
	y^T_{\xi j}= 
	\begin{cases}
		a_{\xi j}^T T_\xi , & \text{for } 1 \leq j \leq |\mathcal A_{\xi}| \\
		\boldsymbol{0}, & \text{for } |\mathcal A_{\xi}| < j \leq |J| 
	\end{cases}
\end{equation*}

\begin{equation*}
	z_{\xi j}= 
	\begin{cases}
		a_{\xi j}^T h_\xi , & \text{for } 1 \leq j \leq |\mathcal A_{\xi}| \\
		{0}, & \text{for } |\mathcal A_{\xi}| < j \leq |J|.
	\end{cases}
\end{equation*}
Then, we can redefine 
\begin{equation}
    H_x=\bigcap_{j=1}^{|J|} \{ \xi : y_{\xi j}^T x \leq z_{\xi j}\}.
\end{equation}
Since $\mathcal X \subseteq \{x\in\mathbb R^n: \|x\|_0 \leq n_0\}$, it follows from \cref{old}, \cref{space} and \cref{sparse} that
\begin{equation*}
    d_{VC}(\mathcal H) \leq \bigg(\frac{e}{(e-1)\log2} |J|\Big(2(n_0+1)\log_2\Big(\frac{ne+e}{n_0+1}\Big)\Big)\cdot \log \big(\frac{e}{\log2}\cdot |J|\big)\bigg).
\end{equation*}
The rest of the proof follows from \cref{main}.
\end{proof}

As we can see, by leveraging the sparsity of the solutions, the sample complexity in \cref{spaa} is improved to $N=O\big(\frac{n_0 \log n}{\epsilon}\log\frac{1}{\epsilon}+\frac{1}{\epsilon}\log\frac{1}{\delta}\big)$ versus the $N=O\big(\frac{ n}{\epsilon}\log\frac{1}{\epsilon}+\frac{1}{\epsilon}\log\frac{1}{\delta}\big)$ in \cref{2-stage}, similar to the trade-off in high-dimensional SAA between \eqref{scomp} and \eqref{newsp}.

\subsection{Special Structures}
In this section, we consider two special structures. The first one is the chain-constrained domain considered in \cite{liu2019feasibility}. The second one is a finite feasible domain considered in \cite{chen2019sample}.

\subsubsection{Chain-Constrained Domain}
In previous sections, \cref{main} is used to analyze example \eqref{ex1} without using the chain-constrained structure as in \cite{liu2019feasibility}. However, it is worth noting that \cref{main} still offers an explicit bound on the feasibility of $x^\star(\xi^{[N]})$ based solely on the chain-constrained structure, although at a slightly worse sample complexity than \cite{liu2019feasibility}. To see this, notice the VC dimension of any chain-constrained domain can be directly bounded.

\begin{lemma}\label{l1}
	If $\text{dom } f_\xi$ has a chain-constrained domain of order $m$, then the VC dimension of $\mathcal H=\{H_x\}_{x\in\mathcal X}\cup\{\Xi\}$ in \eqref{domm} satisfies
	 $$d_{VC}(\mathcal H)\leq\frac{e}{(e-1)\log2} m \log \big(\frac{e}{\log2}\cdot m\big)=O( m\log m).$$
\end{lemma}
\begin{proof}
	
Recall \cref{chain}. Since $\text{dom } f_\xi$ is a chain-constrained domain of order $m$, we can write $\text{dom } f_\xi=\bigcap_{k=1}^m U^\xi_k$ where each $U_k^\xi \in \{U_k^{\xi^\prime} \}_{\xi^\prime\in\Xi}$ is a chain living on $\mathcal X\subseteq \mathbb R^n$ indexed by $\xi\in\mathbb R^r$. Now, for $k\in[m]$, define $W^x_k\triangleq\{\xi: x\in U_k^\xi\}$, and we have from \eqref{domm} that $H_x=\{\xi: x\in \text{ dom} f_\xi\}=\{\xi: x\in\bigcap_{k=1}^m U^\xi_k\}=\bigcap_{k=1}^m W_k^x$. We show, for each $k\in [m]$, $\{W_k^x\}_{x\in\mathcal X}$ is a chain as well. Suppose this is not the case, then there exists $x_1,x_2\in\mathcal X$ such that $W_k^{x_1}\not\subseteq W_k^{x_2}$ and $W_k^{x_2}\not\subseteq W_k^{x_1}$. This implies there exist $\xi_1\in W_k^{x_1}$ and $\xi_2\in W_k^{x_2}$ such that $\xi_1\notin W_k^{x_2} $ and $\xi_2\notin W_k^{x_1} $. This further implies $x_1\in U_k^{\xi_1}, x_2 \notin U_k^{\xi_1}$ and $x_2\in U_k^{\xi_2}, x_1 \notin U_k^{\xi_2}$. Consequently, neither $U_k^{\xi_1}\subseteq U_k^{\xi_2}$ nor $U_k^{\xi_2}\subseteq U_k^{\xi_1}$ is true, contradicting the assumption that $\{U_k^{\xi} \}_{\xi\in\Xi}$ is a chain. Thus, $\{W_k^x\}_{x\in\mathcal X}$ is a chain on $\Xi$ for each $k\in[m]$. It then follows trivially  $\{W_k^x\}_{x\in\mathcal X}\cup\{\Xi\}$ is also a chain on $\Xi$ for each $k\in[m]$. Consequently, $\mathcal H=\{H_x\}_{x\in\mathcal X}\cup\{\Xi\}$ is a chain-constrained domain of order $m$.

On the other hand, the VC dimensions of chains $\{U^\omega\}_{\omega\in I}$ are at most 1 because they cannot shatter any two points. In particular, if $\{x_1,x_2\}$ are two points living on the same space as $\{U^\omega\}_{\omega\in I} $, the shattering of  $\{x_1,x_2\}$ requires $x_1 \in U^{\omega_1}, x_2 \notin U^{\omega_1}$ and $x_2\in U^{\omega_2}, x_1 \notin U^{\omega_2}$ for some $U^{\omega_1},U^{\omega_2} \in \{U^\omega\}_{\omega\in I} $. If this were to happen, then neither $U^{\omega_1}\subseteq U^{\omega_2}$ nor $U^{\omega_2}\subseteq U^{\omega_1}$ would be true, contradicting the definition of a chain. Then, if $\{\mathcal U_k\}_{k\in[m]}$ are the $m$ chains consisting of a chain-constrained domain $\mathcal U$ of order $m$ where each $U\in\mathcal U$ is of the form $U=\bigcap_{k=1}^m U_k$ for some $U_k \in \mathcal U_k$, it again follows from Theorem 1.1 in \cite{van2009note} that 
\begin{equation*}
	d_{VC}(\mathcal U) \leq d_{VC}(\sqcap_{k=1}^m \mathcal U_k) \leq \frac{e}{(e-1)\log2} m \log (\frac{e}{\log2}\cdot m),
\end{equation*}
where $\sqcap_{k=1}^m \mathcal U_k\triangleq \bigg\{\bigcap_{k=1}^m U_k: U_k \in \mathcal U_k, k \in [m]\bigg\}$.  The result follows now from the fact that $\mathcal H$ is a chain of order $m$.
\end{proof}
\cref{l1} combined with \cref{main} can provide an explicit sample complexity for feasibility.

\begin{corollary}\label{f}
If $\text{dom } f_\xi$ has a chain-constrained domain of order $m$, then \cref{main} guarantees that for \begin{equation}
	 N\geq\frac{4}{\epsilon}\bigg(\Big(\frac{e}{(e-1)\log2} m \log \big(\frac{e}{\log2}\cdot m\big)\Big)\log\big(\frac{12}{\epsilon}\big)+\log\big(\frac{2}{\delta}\big)\bigg),
\end{equation}
we have
	\begin{equation*}
\mathbb P^N(V(x^\star (\xi^{[N]}))>\epsilon) \leq  \delta.
	\end{equation*}
for any $0<\delta,\epsilon <1$.
\end{corollary}

\cref{f} provides a sample complexity $O(\frac{m}{\epsilon}\log m \log (\frac{1}{\epsilon})+\frac{1}{\epsilon} \log(\frac{1}{\delta}))$ for chain-constrained domains, or $O(\frac{m}{\epsilon}\log m +\frac{1}{\epsilon} \log(\frac{1}{\delta}))$ using the PAC bounds from \cite{vcnew,vcnew0}, while Scenario 1 in \cite{liu2019feasibility} provides a $O(\frac{m}{\epsilon}+\frac{1}{\epsilon} \log(\frac{1}{\delta}))$ bound according to \eqref{scomple}. As we can see, the more refined analysis on the chain-constrained structure in \cite{liu2019feasibility} leads to a better rate over \cref{f} by log factors. However, the generality offered by \cref{main} is still noteworthy, since its applicability in most situations does not hinge on the chain-constrained domain. 

\subsubsection{Finite Feasible Region}

In this subsection, we apply \cref{main} to the case where the decision set $\mathcal X$ is finite. 

\begin{corollary}\label{finite}
		Suppose $|\mathcal X|<+\infty$ and let $\xi^{[N]}=\{\xi_1,...,\xi_N\}$ be IID samples from $\mathbb P$ (consequently $\xi^{[N]}\sim \mathbb P^N$). Then, if 
	
	\begin{equation}\label{sf}
		N \geq \frac{4}{\epsilon}\Big(\log_2(|\mathcal X|+1)\cdot\log\big(\frac{12}{\epsilon}\big)+\log\big(\frac{2}{\delta}\big)\Big),
	\end{equation}
		we have
	\begin{equation*}
		\mathbb P^N(V(x^\star (\xi^{[N]}))>\epsilon) \leq  \delta
	\end{equation*}
	for any $0<\delta,\epsilon <1$.
\end{corollary}

\begin{proof}
	Let $\mathcal H\triangleq \{H_x\}_{x\in\mathcal X}\cup\{\Xi\}$ be the class of subsets defined in \eqref{domm}. It follows that $|\mathcal H|\leq |\mathcal X|+1<+\infty$. It is known that if $|\mathcal H|<+\infty$, then $d_{VC}(\mathcal H)\leq \log_2 |\mathcal H|$ (by definition of the VC dimension or see \cite{anthony1997computational}). The result then follows from \cref{main}.
\end{proof}

Note that since the VC dimension of a finite hypothesis class is bounded by the logarithm of its cardinality, we get the results in \cref{finite} for free. In Section 4 of \cite{chen2019sample}, the case of finite feasible region $\mathcal X$ is also discussed, with a slightly different focus. In particular, with assumptions on the moment generating functions, \cite{chen2019sample} proves the exponential convergence of a $\delta$-optimal set towards an $\epsilon$-optimal set using large deviations (LD) theory. The rate of convergence also depends on constants from the LD analysis. However, \cite{chen2019sample} also offers a more direct analysis on the feasibility of SAA solution $x^\star(\xi^{[N]})$ when $|\mathcal X|<+\infty$ which does not rely on distributional assumptions of $f(\xi,x)$. To be specific, Lemma 9 of \cite{chen2019sample} states:
\begin{equation}\label{chen}
	\mathbb P^N(\hat F_N(x)<+\infty)\leq (1-\eta)^N, \text { for } x \in\mathcal X^{Infea}
\end{equation}
where $\mathcal X^{Infea}=\{x: x\in\mathcal X \text{ and } V(x)>0\}$ and $\eta=\min\{V(x): x\in\mathcal X^{Infea}\}$. Building on \eqref{chen}, we can deduce the following direct bound regarding $x^\star(\xi^{[N]})$:
\begin{align}\label{fin}
\mathbb P^N(V(x^\star(\xi^{[N]}))>\eta)&=\mathbb P^N(x^\star(\xi^{[N]})\in\mathcal X^{Infea})\nonumber\\
&\leq \mathbb P^N \big(\bigcup_{x\in\mathcal X^{Infea}} \{\hat F_N(x)<+\infty\}\big)\nonumber\\
&\leq \sum_{x\in\mathcal X^{Infea}} \mathbb P^N(\hat F_N(x)<+\infty) \leq |\mathcal X^{Infea}|(1-\eta)^N,
\end{align}
which leads to a $O(\frac{1}{\eta}\log(|\mathcal X|)+\frac{1}{\eta}\log(\frac{1}{\beta}))$ sample complexity, comparable to the $O(\frac{1}{\eta}\log(\frac{1}{\eta})\log_2(|\mathcal X|)+\frac{1}{\eta}\log(\frac{1}{\beta}))$ complexity in \eqref{sf}. Moreover, if we utilize the PAC bound from \cref{t1r}, the bound in \eqref{sf} could be improved to $O(\frac{1}{\eta}\log_2(|\mathcal X|)+\frac{1}{\eta}\log(\frac{1}{\beta}))$ which is of the same order as \eqref{fin}.

 \appendix
 
 
 \section{A Verifiable Example}
In this section, we validate the result of  \cref{3-stage} on an artificial example that is specially constructed so that the optimal SAA solution $x^\star(\xi^{[N]})$ and its violation probability $V(x^\star(\xi^{[N]}))$ can be calculated analytically. This then allows us to directly check the validity of our bound in \cref{3-stage}. Note that bounds obtained from computational learning theory and the VC dimension can be crude for small or midsize problems, and they become more interesting when $N,n $ are large and $\epsilon$ is tiny. However, in the latter large-scale setting, it can be difficult to verify experimentally the quality of the bounds since an accurate estimation of the violation probability is costly or even infeasible. To overcome this difficulty, we build this example that exhibits a closed form so that we can directly check our bound's validity for any combination of $N,n$ (arbitrarily large) and $\epsilon$ (arbitrarily small).


Our example is a modification of the two-stage resource planning (TRP) problem in \cite{chen2019sample} (see also \cite{liu2016decomposition,luedtke2014branch}). The example has the form:
\begin{equation}\label{num_exp}
    \inf_{x\geq 0, x\in\mathbb R^n} F(x) \triangleq c^Tx+\mathbb E [f(\xi,x)],
\end{equation}
where $c> \boldsymbol{0}\in\mathbb R^n$ and
\begin{align}
\begin{aligned}\label{formulation}
		f(\xi,x) \triangleq & \inf_{y\in\mathbb R^n}  q^T y \\
		 \textrm{s.t.} \quad &  y \leq x-\xi,\\
		& y\geq0,    \\
	\end{aligned}
\end{align}
 where $q> \boldsymbol{0}\in\mathbb R^n$ and $\xi\in\mathbb R^n$ with each entry $[\xi]_i$ independently and uniformly distributed on $[0,1]$. We summarize our results into the following corollary.
 
 \begin{corollary}
         Let $\xi^{[N]}=\{\xi_1,...,\xi_N\}$ be IID samples from $\mathbb P$ (consequently $\xi^{[N]}\sim \mathbb P^N$). Then the optimal SAA solution of the example in \eqref{num_exp} satisfies 
         \begin{equation}\label{stat1}
             \mathbb P^N( V(x^\star(\xi^{[N]})>\epsilon) = \sum_{i=0}^{n-1}\frac{(N\log(\frac{1}{1-\epsilon}))^i}{i!} e^{-N\log(\frac{1}{1-\epsilon})}.
         \end{equation}
        Moreover, compared with the bound in \cref{3-stage}, we have 
        \begin{equation}\label{stat2}
       \sum_{k=0}^{n-1}\frac{(N\log(\frac{1}{1-\epsilon}))^k}{k!} e^{-N\log(\frac{1}{1-\epsilon})}\leq 2 \exp \left(-\frac{N\epsilon}{4}\right) \left(\frac{12}{\epsilon}\right)^{d_{VC}}.
        \end{equation}
        for any integer $N,n\geq 1$ and $\epsilon\in(0,1)$.
 \end{corollary}

\begin{proof}
    It is straightforward to check that the optimal solution of the SAA problem satisfies 
    \begin{equation}\label{SAAso}
        [x^\star(\xi^{[N]})]_{i}=\max_{j\in[N]} \text{ } [\xi_j]_i
    \end{equation}
    where $i\in[n]$. Moreover, given $\boldsymbol{0}\leq x \leq \boldsymbol{1} \in\mathbb R^n$, since each entry of $\xi$ is independently and uniformly distributed on $[0,1]$, we have
    \begin{align}\label{vio}
        V(x)=&\mathbb P( \cup_{i\in [n]} \text{ } [\xi]_i>[x]_i) \nonumber\\
        =& 1-\mathbb P( \cap_{i\in [n]} \text{ } [\xi]_i\leq [x]_i) \nonumber\\
        =& 1-\prod_{i\in[n]}[x]_i.
    \end{align}
On the other hand, it is known that if $U$ is a uniform random variable on $[0,1]$, then $-\log(U)\sim exp(1)$ follows an exponential distribution with parameter $1$. Moreover, it is known that the minimum of a series of independent exponential distribution $\{V_j\sim exp(\lambda_j)\}_{j\in[N]}$ follows another exponential distribution $exp (\sum_{j\in[N]}\lambda_j)$. Thus, if $\{U_j\}_{j\in[N]}$ is a series of independent and uniform random variables on $[0,1]$, then 
\begin{equation}\label{exp}
    -\log(\max_{j\in[N]} U_j)=\min_{j\in[N]} -\log U_j \sim exp(N)
\end{equation}
follows an exponential distribution with parameter $N$. Finally, it is also known that if we have $n$ independent exponential random variables $\{W_i\}_{i\in[n]}$ with parameter $N$, then their sum:
\begin{equation}\label{gamma}
    \sum_{i\in[n]} W_i \sim gamma (n,N)
\end{equation}
follows a gamma distribution with parameter $(n,N)$ and cumulative distribution function:
\begin{equation}\label{gammacdf}
    F(x; n,N)=1-\sum_{i=0}^{n-1}\frac{(Nx)^i}{i!} e^{-Nx}.
\end{equation}
Now, it follows from \eqref{SAAso} and  \eqref{vio} that
\begin{align*}
    \mathbb P^{N} (V(x^\star(\xi^{[N]}))>\epsilon)=&\mathbb P^{N}( 1-\prod_{i\in[n]}\max_{j\in[N]}[\xi_j]_i > \epsilon) \nonumber\\
    =&\mathbb P^N (\sum_{i\in[n]}-\log(\max_{j\in[N]}[\xi_j]) > \log(\frac{1}{1-\epsilon}))\nonumber\\
    =&\sum_{i=0}^{n-1}\frac{(N\log(\frac{1}{1-\epsilon}))^i}{i!} e^{-N\log(\frac{1}{1-\epsilon})}.
\end{align*}
This proves \eqref{stat1}. To prove the \eqref{stat2}, note that since $\log(\frac{1}{1-\epsilon})>\epsilon$ for $\epsilon \in (0,1)$, it follows from \eqref{exp}, \eqref{gamma} and \eqref{gammacdf} that 
\begin{align*}
    \mathbb P^{N} (V(x^\star(\xi^{[N]}))>\epsilon)=&\mathbb P^N (\sum_{i\in[n]}-\log(\max_{j\in[N]}[\xi_j]) > \log(\frac{1}{1-\epsilon}))\nonumber\\
    \leq & \mathbb P^N (\sum_{i\in[n]}-\log(\max_{j\in[N]}[\xi_j]) >\epsilon) \nonumber\\
    =& 1-F(\epsilon; n,N) \nonumber\\
    = &\sum_{i=0}^{n-1}\frac{(N\epsilon)^i}{i!} e^{-N\epsilon}.
\end{align*}
Now, to prove \eqref{stat2}, it suffices to prove $$2\exp \left(-\frac{N\epsilon}{4}\right) \left(\frac{12}{\epsilon}\right)^{d_{VC}}\geq \sum_{i=0}^{n-1}\frac{(N\epsilon)^i}{i!} e^{-N\epsilon}.$$ First note from the formulation \eqref{formulation} and \eqref{dvc} that $d_{VC}\geq n$. Then, if we define $\theta\triangleq N\epsilon$, it suffices to show
\begin{equation}\label{last}
    2\left(\frac{12}{\epsilon}\right)^{n} \geq h(\theta; n)\triangleq \sum_{i=0}^{n-1}\frac{(\theta)^i}{i!} e^{-\frac{3}{4}\theta}
\end{equation}
for any integer $n\geq 1$, $\theta\geq 0$ and $\epsilon\in(0,1)$. We prove \eqref{last} by contradiction. First note \eqref{last} is trivially true when $n=1$. Suppose \eqref{last} is not true, then there exists $n\geq 2$ and $\epsilon\in(0,1)$ such that 
\begin{equation}\label{llast}
    \sup_{\theta\geq 0} h(\theta; n) > 2\left(\frac{12}{\epsilon}\right)^{n}
\end{equation}
However, when $\theta\in[0,1]$, it is clear that 
\begin{equation*}
    \sup_{\theta\geq 0} h(\theta; n) \leq \sum_{i=0}^{n-1}\frac{1}{i!} e^{0} \leq e < 2\left(\frac{12}{\epsilon}\right) \leq 2\left(\frac{12}{\epsilon}\right)^n
\end{equation*}
On the other hand, when $\theta\rightarrow\infty$, the exponential decay dominates and we have $\lim_{\theta\rightarrow\infty}h(\theta; n)\rightarrow 0$. Consequently, there must exists some $\theta_{u}$ such that 
\begin{equation*}
    \sup_{\theta\geq \theta_{u}} h(\theta; n) \leq 2\left(\frac{12}{\epsilon}\right)^n.
\end{equation*}
Combined with \eqref{llast}, the above inequalities implies $ \sup_{1\leq \theta \leq \theta_{u}} h(\theta; n) > 2\left(\frac{12}{\epsilon}\right)^n$. By the extreme value theorem, there exists a local maximum $\theta^\star\in[1,\theta_u]$ such that
\begin{equation}\label{lllast}
    h(\theta^\star ; n)> 2\left(\frac{12}{\epsilon}\right)^n,
\end{equation}
and $h'(\theta^\star ; n)=0$. Solving for $h'(\theta^\star ; n)=0$, we obtain that
\begin{equation*}
    \sum_{i=0}^{n-2}\frac{(\theta^\star)^i}{i!} =\frac{3(\theta^\star)^{n-1}}{(n-1)!} \quad\text{ and }\quad h(\theta^\star; n) = \frac{4(\theta^\star)^{n-1}}{(n-1)!}e^{-\frac{3}{4}\theta^\star}. 
\end{equation*}
However, it can be checked that, for any integer $m\geq 1$, 
\begin{equation*}
 \max_{\theta\geq 0} \frac{4\theta^{m}}{m!}e^{-\frac{3}{4}\theta}=4(\frac{4}{3})^m\frac{m^m}{m!}e^{-m}.
\end{equation*}
 Then, using Stirling’s formula (see \cite{conrad2016stirling}) we have $m!\geq {m^m}{e^{-m}}\sqrt{2\pi m}$ and
 \begin{equation}\label{inter}
     \max_{\theta\geq 0} \frac{4\theta^{m}}{m!}e^{-\frac{3}{4}\theta} \leq 2\left(\frac{4}{3}\right)^m<2\left(\frac{12}{\epsilon}\right)^m.
 \end{equation}
 Now, using the result in \eqref{inter}, we have 
\begin{equation*}
    h(\theta^\star ; n) \leq \max_{\theta\geq 0} \frac{4\theta^{n-1}}{(n-1)!}e^{-\frac{3}{4}\theta}<2\left(\frac{12}{\epsilon}\right)^{n-1}<2\left(\frac{12}{\epsilon}\right)^n
\end{equation*}
which contradicts with \eqref{lllast}. This concludes the proof.
\end{proof}

\section{Technical Discussion on Measurability}

As discussed in \cite{korf2001random} or \cite{shapiro2014lectures}, an extended real valued function $f(\xi(\omega),x):\Omega\times\mathcal X\rightarrow\mathbb R\cup \{\infty\}$ defined on $\Omega\times\mathcal X$ equipped with a $\mathbb P$-complete measure and $\sigma$-algebra $\mathcal F \otimes \mathscr B$ where $\mathcal X$ is a Polish space (a complete, separable, metric space) is a random lower semicontinuous function if and only if 
\begin{enumerate}
    \item $f$ is $\mathcal F \otimes \mathscr B$-measurable,
    \item for every $\omega\in\Omega$, $f(\xi(\omega),\cdot)$ is lower semicontinuous (i.e., $\liminf_{x\rightarrow x_0} f(\xi,x) \geq f(\xi,x_0)$ for all $x_0\in\mathcal X$),
\end{enumerate}
which is the condition specified in the introduction following \eqref{sp}. Note that when $f: \mathcal X \rightarrow \mathbb R$, a notable class of random lower semicontinuous functions is given by the Carath\'eodory functions (also a standard setup in SAA literature \cite{shapiro2014lectures}), which simply requires $f$'s continuity in $x$ and measurability in $\xi$. However, for $f: \mathcal X \rightarrow \mathbb R\cup\{\infty\}$, lower semicontinuity in $x$ and measurability in $\xi$ are not enough for $f$ to be random lower semicontinuous. One must also require $f$ to be $\mathcal F\otimes\mathscr B$-measurable \cite{shapiro2014lectures,korf2001random}.

Given some $N \geq 1$ and the random lower semicontinuous function $f$, the sum of lower semicontinuous $\hat F_N$ is also a random lower semicontinuous since the sum is always well-defined (there is no indeterminate form $-\infty+\infty$). Consequently, as shown in Theorem 7.36 and Theorem 7.37 of \cite{shapiro2014lectures}, the optimal solution $x^\star(\xi^{[N]})$ is measurable. On the other hand, fixing any $x_0\in\mathcal X$, by Fatou's lemma
\begin{align*}
    \liminf_{x\rightarrow x_0}V(x)=\liminf_{x\rightarrow x_0}\mathbb E \mathbf{1}\{f(\xi,x)=\infty\} \geq &\mathbb E \liminf_{x\rightarrow x_0} \mathbf{1}\{f(\xi,x)=\infty\}\nonumber\\
    \geq & \mathbb E \mathbf{1}\{f(\xi,x_0)=\infty\}=V(x_0),
\end{align*}
where the last inequality follows from the lower semicontinuity of $\mathbf{1}\{f(\xi,\cdot)=\infty\}$ due to the lower semicontinuity of $f(\xi,\cdot)$ and the non-negativity of the indicator function. Thus, $V(\cdot)$ is a lower semicontinuous thus $\mathscr B$-measurable function on $\mathcal X$. It then follows $V(x^\star(\xi^{[N]}))$ is $\mathcal F$-measurable, suggesting \eqref{mea2} is well-defined. 

For \eqref{meaissue}, we first note that for any $\epsilon>0$, the set $\{\sup_{x\in\text{dom } \hat F_N} V(x) > \epsilon\}$ is a.e. equivalent to $\big\{\sup_{x\in\mathcal X}\big( V(x)\cdot\prod_{i=1}^N \mathbf{1}\{f(\xi_i,x)<\infty\}\big)  > \epsilon\big\}$. It follows from the $\mathcal F \otimes \mathscr B$ measurability of $f$ and the fact that product of non-negative lower semicontinuous functions (i.e., $\mathbf{1}\{f(\xi,\cdot)=\infty\}$) is also lower semicontinuous, that 
\begin{equation}\label{lcv}
    V(x)\cdot\prod_{i=1}^N \mathbf{1}\{f(\xi_i,x)<\infty\}
\end{equation}
 is a random lower semicontinuous function. Thus, fixing any $x\in\mathcal X$, the set $\big\{\big( V(x)\cdot\prod_{i=1}^N \mathbf{1}\{f(\xi_i,x)<\infty\}\big)  > \epsilon\big\}$ is $\mathbb P$-measurable.   Now, let $\{x_j\}_{j\in\mathbb N}$ be a countable dense subset of $\mathcal X$, the lower semicontinuity of \eqref{lcv} implies $$\big\{\sup_{x\in\mathcal X}\big( V(x)\cdot\prod_{i=1}^N \mathbf{1}\{f(\xi_i,x)<\infty\}\big)  > \epsilon\big\}$$ is a.e.  equivalent to
 \begin{equation*}
     \big\{\sup_{j\in\mathbb N}\big( V(x_j)\cdot\prod_{i=1}^N \mathbf{1}\{f(\xi_i,x_j)<\infty\}\big)  > \epsilon\big\}=\bigcup_{j\in\mathbb N}\big\{\big( V(x_j)\cdot\prod_{i=1}^N \mathbf{1}\{f(\xi_i,x_j)<\infty\}\big)  > \epsilon\big\},
 \end{equation*}
which is $\mathbb P$-measurable. The completeness of $\mathbb P$ now suggests the measurability of \eqref{meaissue}.

\section*{Acknowledgments}
We greatly thank Alexander Shapiro for introducing to us the motivating problem and the helpful suggestions, and James Luedtke for the helpful feedback in the preparation of this manuscript.

\bibliographystyle{siamplain}
\bibliography{references}

\begin{thebibliography}{10}

\bibitem{abramovich2018high}
{\sc F.~Abramovich and V.~Grinshtein}, {\em High-dimensional classification by
  sparse logistic regression}, IEEE Transactions on Information Theory, 65
  (2018), pp.~3068--3079.

\bibitem{ahmed2010two}
{\sc S.~Ahmed}, {\em Two-stage stochastic integer programming: A brief
  introduction}, Wiley encyclopedia of operations research and management
  science,  (2010).

\bibitem{ahmed2002sample}
{\sc S.~Ahmed, A.~Shapiro, and E.~Shapiro}, {\em The sample average
  approximation method for stochastic programs with integer recourse},
  Submitted for publication,  (2002), pp.~1--24.

\bibitem{anthony1997computational}
{\sc M.~Anthony and N.~Biggs}, {\em Computational learning theory}, vol.~30,
  Cambridge University Press, 1997.

\bibitem{barbarosoglu2004two}
{\sc G.~Barbarosoǧlu and Y.~Arda}, {\em A two-stage stochastic programming
  framework for transportation planning in disaster response}, Journal of the
  operational research society, 55 (2004), pp.~43--53.

\bibitem{bertsimas2016best}
{\sc D.~Bertsimas, A.~King, and R.~Mazumder}, {\em Best subset selection via a
  modern optimization lens}, The annals of statistics, 44 (2016), pp.~813--852.

\bibitem{bidhandi2017accelerated}
{\sc H.~M. Bidhandi and J.~Patrick}, {\em Accelerated sample average
  approximation method for two-stage stochastic programming with binary
  first-stage variables}, Applied Mathematical Modelling, 41 (2017),
  pp.~582--595.

\bibitem{blumer1989learnability}
{\sc A.~Blumer, A.~Ehrenfeucht, D.~Haussler, and M.~K. Warmuth}, {\em
  Learnability and the vapnik-chervonenkis dimension}, Journal of the ACM
  (JACM), 36 (1989), pp.~929--965.

\bibitem{borosh1976bounds}
{\sc I.~Borosh and L.~B. Treybig}, {\em Bounds on positive integral solutions
  of linear diophantine equations}, Proceedings of the American Mathematical
  Society, 55 (1976), pp.~299--304.

\bibitem{bugg2021logarithmic}
{\sc C.~Bugg and A.~Aswani}, {\em Logarithmic sample bounds for sample average
  approximation with capacity-or budget-constraints}, Operations Research
  Letters, 49 (2021), pp.~231--238.

\bibitem{calafiore2005uncertain}
{\sc G.~Calafiore and M.~C. Campi}, {\em Uncertain convex programs: randomized
  solutions and confidence levels}, Mathematical Programming, 102 (2005),
  pp.~25--46.

\bibitem{campi2008exact}
{\sc M.~C. Campi and S.~Garatti}, {\em The exact feasibility of randomized
  solutions of uncertain convex programs}, SIAM Journal on Optimization, 19
  (2008), pp.~1211--1230.

\bibitem{campi2011sampling}
{\sc M.~C. Campi and S.~Garatti}, {\em A sampling-and-discarding approach to
  chance-constrained optimization: feasibility and optimality}, Journal of
  Optimization Theory and Applications, 148 (2011), pp.~257--280.

\bibitem{care2014fast}
{\sc A.~Car{\`e}, S.~Garatti, and M.~C. Campi}, {\em Fast—fast algorithm for
  the scenario technique}, Operations Research, 62 (2014), pp.~662--671.

\bibitem{chen2019sample}
{\sc R.~Chen and J.~Luedtke}, {\em On sample average approximation for
  two-stage stochastic programs without relatively complete recourse}, arXiv
  preprint arXiv:1912.13078,  (2019).

\bibitem{chen2019convergence}
{\sc X.~Chen, A.~Shapiro, and H.~Sun}, {\em Convergence analysis of sample
  average approximation of two-stage stochastic generalized equations}, SIAM
  Journal on Optimization, 29 (2019), pp.~135--161.

\bibitem{conrad2016stirling}
{\sc K.~Conrad}, {\em Stirling’s formula}, Available in http://www. math.
  uconn. edu/kconrad/blu rbs/analysis/stirling. pdf,  (2016).

\bibitem{de2004constraint}
{\sc D.~P. De~Farias and B.~Van~Roy}, {\em On constraint sampling in the linear
  programming approach to approximate dynamic programming}, Mathematics of
  operations research, 29 (2004), pp.~462--478.

\bibitem{dillon2017two}
{\sc M.~Dillon, F.~Oliveira, and B.~Abbasi}, {\em A two-stage stochastic
  programming model for inventory management in the blood supply chain},
  International Journal of Production Economics, 187 (2017), pp.~27--41.

\bibitem{dudley1978central}
{\sc R.~M. Dudley}, {\em Central limit theorems for empirical measures}, The
  Annals of Probability,  (1978), pp.~899--929.

\bibitem{erdougan2006ambiguous}
{\sc E.~Erdo{\u{g}}an and G.~Iyengar}, {\em Ambiguous chance constrained
  problems and robust optimization}, Mathematical Programming, 107 (2006),
  pp.~37--61.

\bibitem{ermoliev2013sample}
{\sc Y.~M. Ermoliev and V.~I. Norkin}, {\em Sample average approximation method
  for compound stochastic optimization problems}, SIAM Journal on Optimization,
  23 (2013), pp.~2231--2263.

\bibitem{vcnew}
{\sc S.~Hanneke}, {\em The optimal sample complexity of pac learning}, Journal
  of Machine Learning Research, 17 (2016), pp.~1--15,
  \url{http://jmlr.org/papers/v17/15-389.html}.

\bibitem{higle2013stochastic}
{\sc J.~L. Higle and S.~Sen}, {\em Stochastic decomposition: a statistical
  method for large scale stochastic linear programming}, vol.~8, Springer
  Science \& Business Media, 2013.

\bibitem{huang2000inexact}
{\sc G.~Huang and D.~P. Loucks}, {\em An inexact two-stage stochastic
  programming model for water resources management under uncertainty}, Civil
  Engineering Systems, 17 (2000), pp.~95--118.

\bibitem{jing2020complexity}
{\sc R.-J. Jing, M.~Moreno-Maza, and D.~Talaashrafi}, {\em Complexity estimates
  for fourier-motzkin elimination}, in International Workshop on Computer
  Algebra in Scientific Computing, Springer, 2020, pp.~282--306.

\bibitem{jirutitijaroen2008reliability}
{\sc P.~Jirutitijaroen and C.~Singh}, {\em Reliability constrained multi-area
  adequacy planning using stochastic programming with sample-average
  approximations}, IEEE Transactions on Power Systems, 23 (2008), pp.~504--513.

\bibitem{kaariainen2004relating}
{\sc M.~K{\"a}{\"a}ri{\"a}inen}, {\em Relating the rademacher and vc bounds},
  tech. report, Citeseer, 2004.

\bibitem{kearns1994introduction}
{\sc M.~J. Kearns, U.~V. Vazirani, and U.~Vazirani}, {\em An introduction to
  computational learning theory}, MIT press, 1994.

\bibitem{korf2001random}
{\sc L.~A. Korf and R.~J.-B. Wets}, {\em Random lsc functions: An ergodic
  theorem}, Mathematics of Operations Research, 26 (2001), pp.~421--445.

\bibitem{kuccukyavuz2017introduction}
{\sc S.~K{\"u}{\c{c}}{\"u}kyavuz and S.~Sen}, {\em An introduction to two-stage
  stochastic mixed-integer programming}, in Leading Developments from INFORMS
  Communities, INFORMS, 2017, pp.~1--27.

\bibitem{linderoth2006empirical}
{\sc J.~Linderoth, A.~Shapiro, and S.~Wright}, {\em The empirical behavior of
  sampling methods for stochastic programming}, Annals of Operations Research,
  142 (2006), pp.~215--241.

\bibitem{liu2009two}
{\sc C.~Liu, Y.~Fan, and F.~Ord{\'o}{\~n}ez}, {\em A two-stage stochastic
  programming model for transportation network protection}, Computers \&
  Operations Research, 36 (2009), pp.~1582--1590.

\bibitem{liu2019regularized}
{\sc H.~Liu, C.~Hernandez, and H.~Y. Lee}, {\em Regularized sample average
  approximation for high-dimensional stochastic optimization under
  low-rankness}, arXiv preprint arXiv:1904.03453,  (2019).

\bibitem{liu2019sample}
{\sc H.~Liu, X.~Wang, T.~Yao, R.~Li, and Y.~Ye}, {\em Sample average
  approximation with sparsity-inducing penalty for high-dimensional stochastic
  programming}, Mathematical programming, 178 (2019), pp.~69--108.

\bibitem{liu2020asymptotic}
{\sc J.~Liu and S.~Sen}, {\em Asymptotic results of stochastic decomposition
  for two-stage stochastic quadratic programming}, SIAM Journal on
  Optimization, 30 (2020), pp.~823--852.

\bibitem{liu2019feasibility}
{\sc R.~P. Liu}, {\em On feasibility of sample average approximation
  solutions}, SIAM Journal on Optimization, 30 (2020), pp.~2026--2052.

\bibitem{liu2016decomposition}
{\sc X.~Liu, S.~K{\"u}{\c{c}}{\"u}kyavuz, and J.~Luedtke}, {\em Decomposition
  algorithms for two-stage chance-constrained programs}, Mathematical
  Programming, 157 (2016), pp.~219--243.

\bibitem{luedtke2014branch}
{\sc J.~Luedtke}, {\em A branch-and-cut decomposition algorithm for solving
  chance-constrained mathematical programs with finite support}, Mathematical
  Programming, 146 (2014), pp.~219--244.

\bibitem{luedtke2008sample}
{\sc J.~Luedtke and S.~Ahmed}, {\em A sample approximation approach for
  optimization with probabilistic constraints}, SIAM Journal on Optimization,
  19 (2008), pp.~674--699.

\bibitem{noyan2012risk}
{\sc N.~Noyan}, {\em Risk-averse two-stage stochastic programming with an
  application to disaster management}, Computers \& Operations Research, 39
  (2012), pp.~541--559.

\bibitem{oliveira2017sample1}
{\sc R.~I. Oliveira and P.~Thompson}, {\em Sample average approximation with
  heavier tails i: non-asymptotic bounds with weak assumptions and stochastic
  constraints}, arXiv preprint arXiv:1705.00822,  (2017).

\bibitem{oliveira2017sample2}
{\sc R.~I. Oliveira and P.~Thompson}, {\em Sample average approximation with
  heavier tails ii: localization in stochastic convex optimization and
  persistence results for the lasso}, arXiv preprint arXiv:1711.04734,  (2017).

\bibitem{pagnoncelli2009sample}
{\sc B.~K. Pagnoncelli, S.~Ahmed, and A.~Shapiro}, {\em Sample average
  approximation method for chance constrained programming: theory and
  applications}, Journal of optimization theory and applications, 142 (2009),
  pp.~399--416.

\bibitem{Panik1993}
{\sc M.~J. Panik}, {\em Extreme Points and Directions for Convex Sets},
  Springer Netherlands, Dordrecht, 1993, pp.~189--234,
  \url{https://doi.org/10.1007/978-94-015-8124-0_8},
  \url{https://doi.org/10.1007/978-94-015-8124-0_8}.

\bibitem{shapiro2014lectures}
{\sc A.~Shapiro, D.~Dentcheva, and A.~Ruszczy{\'n}ski}, {\em Lectures on
  Stochastic Programming: Modeling and Theory}, SIAM, 2014.

\bibitem{shapiro2005complexity}
{\sc A.~Shapiro and A.~Nemirovski}, {\em On complexity of stochastic
  programming problems}, in Continuous optimization, Springer, 2005,
  pp.~111--146.

\bibitem{vcnew0}
{\sc H.~U. Simon}, {\em An almost optimal pac algorithm}, in Proceedings of The
  28th Conference on Learning Theory, P.~Grünwald, E.~Hazan, and S.~Kale,
  eds., vol.~40 of Proceedings of Machine Learning Research, Paris, France,
  03--06 Jul 2015, PMLR, pp.~1552--1563,
  \url{http://proceedings.mlr.press/v40/Simon15a.html}.

\bibitem{talagrand2014upper}
{\sc M.~Talagrand}, {\em Upper and lower bounds for stochastic processes:
  modern methods and classical problems}, vol.~60, Springer Science \& Business
  Media, 2014.

\bibitem{terzer2009large}
{\sc M.~Terzer}, {\em Large scale methods to enumerate extreme rays and
  elementary modes}, PhD thesis, ETH Zurich, 2009.

\bibitem{van2009note}
{\sc A.~Van Der~Vaart and J.~A. Wellner}, {\em A note on bounds for vc
  dimensions}, Institute of Mathematical Statistics collections, 5 (2009),
  p.~103.

\bibitem{van1996weak}
{\sc A.~W. Van Der~Vaart and J.~A. Wellner}, {\em Weak convergence}, in Weak
  convergence and empirical processes, Springer, 1996, pp.~16--28.

\bibitem{vapnik2013nature}
{\sc V.~Vapnik}, {\em The nature of statistical learning theory}, Springer
  science \& business media, 2013.

\bibitem{wang2008sample}
{\sc W.~Wang and S.~Ahmed}, {\em Sample average approximation of expected value
  constrained stochastic programs}, Operations Research Letters, 36 (2008),
  pp.~515--519.

\end{thebibliography}
\end{document}